\documentclass[11pt,twoside,draft,final]{article}
\usepackage{amsmath,amsfonts,amssymb,amsthm,indentfirst,enumerate,textcomp}
\usepackage[utf8]{inputenc}
\usepackage{chebsb}
\usepackage{tikz}
\usepackage[russian]{babel}
\usepackage{indentfirst, array}
\usepackage{amscd,latexsym}
\usepackage{mathrsfs}
\usepackage{tabularx}
\usepackage{multirow}

\theoremstyle{plain}
\newtheorem{thm}{Теорема}[section]
\newtheorem{ass}[thm]{Утверждение}
\newtheorem{lem}[thm]{Лемма}
\newtheorem{cor}[thm]{Следствие}

\theoremstyle{definition}
\newtheorem{dfn}[thm]{Определение}
\newtheorem{rk}[thm]{Замечание}

\usepackage{graphicx}
\graphicspath{{Pictures/}}
\DeclareGraphicsExtensions{.png, .jpg}

\usepackage{textcase}

\def\HP{\operatorname{HP}}
\def\Int{\operatorname{Int}}
\def\Cl{\operatorname{Cl}}
\def\Conv{\operatorname{Conv}}

\label{beg}

\levkolonttl{
Галстян А. Х.}
\prvkolonttl{
Устойчивость границы в проблеме Ферма--Штейнера \ldots}

\title
{
Устойчивость границы в проблеме Ферма--Штейнера в гиперпространствах над конечномерными нормированными пространствами\footnote{Работа выполнена при поддержке гранта РНФ (проект 21-11-00355) в МГУ имени М.~В.~Ломоносова, а также при поддержке гранта Фонда развития теоретической физики и математики ``БАЗИС'' (договор No.~21-8-3-3-1 от 01.10.2021).}}
{
Boundary stability in the Fermat--Steiner problem in hyperspaces over finite-dimensional normed spaces}

\author
{
Галстян А. Х. (г. Москва)}
{
Galstyan A. Kh. (Moscow)}

\info
{
\noindent {\bf Галстян А. Х.}~Московский государственный университет имени М.~В.~Ломоносова

\noindent
\emph{e-mail: ares.1995@mail.ru}

}
{
\noindent {\bf Galstyan A. Kh.}~Lomonosov Moscow State University

\noindent
\emph{e-mail: ares.1995@mail.ru}

}

\Abstract
{
Проблема Ферма--Штейнера состоит в поиске всех точек метрического пространства $Y$ таких, что сумма расстояний от каждой из них до точек из некоторого фиксированного конечного подмножества $A = \{A_1, \ldots, A_n\}$ пространства $Y$ минимальна. Эта проблема рассматривается в случае, когда $Y=\mathcal{H}(X)$ --- это пространство непустых компактных подмножеств конечномерного нормированного пространства $X$, наделённое метрикой Хаусдорфа, то есть $\mathcal{H}(X)$ является \textit{гиперпространством} над $X$. Множество $A$ называют \textit{границей}, все $A_i$ --- \textit{граничными множествами}, а компакты, которые реализуют минимум суммы расстояний до $A_i$ --- компактами Штейнера. 

В данной работе изучается вопрос \textit{устойчивости} в проблеме Ферма--Штейнера при переходе от границы из конечных компактов к границе, состоящей из их выпуклых оболочек. Под устойчивостью здесь имеется в виду, что при переходе к выпуклым оболочкам граничных компактов минимум суммы расстояний не изменится. 

В работе было продолжено изучение геометрических объектов, а именно, \textit{множеств сцепки}, возникающих в проблеме Ферма--Штейнера. Полученные результаты раскрыли некоторую геометрию взаимосвязи компактов Штейнера с граничными множествами. На её основе было выведено достаточное условие неустойчивости границы $\mathcal{H}(X)$.
}
{
The Fermat--Steiner problem is to find all points of the metric space $Y$ such that the sum of the distances from each of them to points from some fixed finite subset $A = \{A_1, \ldots, A_n\}$ of the space $Y$ is minimal . This problem is considered in the case when $Y=\mathcal{H}(X)$ is the space of non-empty compact subsets of a finite-dimensional normed space $X$ endowed with the Hausdorff metric, i.e. $\mathcal{H}(X) $ is a \textit{hyperspace} over $X$. The set $A$ is called \textit{boundary}, all $A_i$ are called \textit{boundary sets}, and the compact sets that realize the minimum of the sum of distances to $A_i$ are called Steiner compacts.

In this paper, we study the question of \textit{stability} in the Fermat--Steiner problem when passing from a boundary consisting of finite compact sets to a boundary consisting of their convex hulls. By stability here we mean that the minimum of the sum of distances does not change when passing to convex hulls of boundary compact sets.

The paper continued the study of geometric objects, namely, \textit{hook sets} that arise in the Fermat--Steiner problem. The results obtained revealed some geometry of the relationship between Steiner compacts and boundary sets. On its basis, a sufficient condition for the instability of the boundary in $\mathcal{H}(X)$ was derived.
}

{
metric geometry, convex sets, Hausdorff distance, continuous deformations, minimal networks, Fermat--Steiner problem, Steiner problem.
}

\Bibliography{22 названий.}{22 titles.}

\begin{document}

\maketitle

\enmaketitle

\section{Введение}

Во многих областях деятельности человека возникают вопросы минимизации затрат каких бы то ни было ресурсов, например, времени в пути на доставку чего-либо: чем оно меньше, тем больше выгода. Проблема, которой посвящена данная работа, как раз из этой области. Одна из частных прикладных постановок задачи может быть такой: допустим, есть ряд предприятий и ставится вопрос о постройке ещё одного дополнительного, которое ежедневно будет тесно взаимодействовать с исходными. Вопрос --- где его разместить? Логично, например, в месте, которое находится на наименьшем суммарном расстоянии от исходных предприятий.

История подобного рода задач восходит ещё к 17 веку (подробную историческую сводку можно найти, например, в работах~\cite{Hist1, Hist2, Hist3}). А именно, в одной из своих работ Пьер Ферма сформулировал следующую проблему. Дано три точки на плоскости, требуется найти четвёртую такую, что сумма расстояний от неё до трёх изначальных минимальна. Около 1640 года Эванджелиста Торричелли, узнав об этой задаче от самого Ферма, решил её, опираясь на законы физики. Это решение было опубликовано в одной из работ Винченцо Вивиани, ученика Торричелли. До сих пор остаётся неясным вопрос о том, знал ли сам Ферма ответ на свою задачу. Традиционная историческая линия отдаёт пальму первенства в решении проблемы именно итальянскому учёному. Стоит отметить, что на протяжении веков задача Ферма забывалась, переоткрывалась и решалась заново разными европейскими математиками посредством новых методов и подходов (Т.~Симпсон, Ф.~Хейнен, Ж.~Бертран и др.). 

Ещё одна похожая проблема была сформулирована в 1836 году Карлом Гауссом. В письме своему ученику Г.~Х.~Шумахеру он задался вопросом, как построить сеть железных дорог минимальной суммарной длины, которая будет соединять четыре немецких города: Бремен, Гамбург, Ганновер и Брауншвейг. В Германии того времени весьма активно развивалось строительство железнодорожной сети по всей стране. Гаусс в своём ответе пишет, что задача Ферма здесь, а именно, поиск точки, реализующей минимум суммы расстояний до четырёх заданных, не приводит к графу минимального веса, соединяющему эти точки, в отличие от случая проблемы на трёх точках. Карл Гаусс в рамках той переписки впервые заявил о хорошо известном сегодня алгоритме построения кратчайшей сети на плоскости, соединяющей конечный набор точек: нужно перебрать все так называемые \textit{графы Штейнера}, соединяющие эти точки~\cite{Hist2, Hwang}. Минимальные сети будут среди них и только среди них. Графом Штейнера принято называть плоский граф, вершины в котором имеют степени смежности только три, два или один. И если вершина имеет степень три, то все рёбра из неё выходят под углом 120 градусов относительно друг друга, а если вершина имеет степень два, то угол между двумя исходящими рёбрами должен быть не меньше 120 градусов.

В 1934 году В. Ярник и М. Кёсслер сформулировали задачи Ферма и Гаусса в более общем виде: требуется соединить $n$ точек в евклидовой плоскости сетью минимального веса~\cite{YaKos}. Широкую известность эта проблема приобрела с выходом в 1941 году книги Р.~Куранта и Г.~Роббинса ``Что такое математика?''~\cite{Steiner}, которая впоследствии оказалась весьма популярной. Однако в этой работе авторы приписывают пальму первенства постановки и решения проблемы построения минимальной сети другому математику, Якобу Штейнеру, работавшему в 19 веке в стенах Берлинского университета. Штейнер формулировал задачу следующим образом: требуется соединить три деревни $A, B$ и $C$ системой дорог минимальной суммарной длины. Как пишут современные историки, всё дело в том, что Курант и Роббинс узнали про задачу Ферма из опубликованных черновиков берлинского математика. Таким образом, с 40-х годов прошлого столетия проблему поиска кратчайшей сети в метрическом пространстве, соединяющей заданные точки, принято называть \textit{проблемой Штейнера}. А родственную задачу, в которой поиск ведётся лишь среди графов типа звезда, где исходный набор точек является множеством листьев такого графа, с недавнего времени стали именовать \textit{проблемой Ферма--Штейнера}~\cite{ITT}.

Общая и естественная постановка задачи может быть такой. Пусть $Y$ --- метрическое пространство и $A$ --- конечное подмножество $Y$. Проблема Ферма--Штейнера состоит в поиске всех точек $y \in Y$, находящихся на наименьшем суммарном расстоянии от точек множества $A = \{A_1, \ldots, A_n\}.$ В соответствии с терминологией из теории графов множество $A$ называют \textit{границей}. Множество всех решений всюду далее будет обозначаться через $\Sigma(A).$ 

В настоящей работе рассматривается проблема Ферма--Штейнера в случае, когда $Y=\mathcal{H}(X)$ является пространством с метрикой Хаусдорфа $d_H,$ см. раздел~\ref{DistSec}, всех непустых компактных подмножеств конечномерного нормированного пространства $X$ над полем $\mathbb{R}$. Пространство $Y=\mathcal{H}(X)$ ещё называют \textit{гиперпространством} над $X,$ см.~\cite{Nadler}. Таким образом, задачу можно сформулировать так. Требуется найти все $K\in \mathcal{H}(X),$ которые реализуют минимум следующего функционала:
\begin{equation}\label{S_A}
S(A, K) = \sum\limits_{i=1}^n d_H(A_i, K).
\end{equation} 
Само минимальное значение $S(A, K)$ будет далее обозначаться через $S_A.$

Вообще говоря, в произвольном метрическом пространстве множество $\Sigma(A)$ может оказаться пустым. Однако из~\cite{ITT} известно, что в ограниченно компактных метрических пространствах, то есть в таких, где всякий замкнутый шар есть компакт, решение проблемы Ферма--Штейнера всегда существует. Более того, в~\cite{ITT} было доказано, что если $X$ ограниченно компактно, то и $\mathcal{H}(X)$ тоже ограниченно компактно. Хорошо известно, что в конечномерных пространствах замкнутое и ограниченное подмножество есть компакт. Таким образом, согласно сказанному выше, в данной работе всюду выполнено $\Sigma(A)\neq \emptyset.$ Элементы из $\Sigma(A)$ далее будут называться \textit{компактами Штейнера}.

Пусть $K\in \Sigma(A).$ Тогда обозначим расстояние по Хаусдорфу между $K$ и $A_i\in A$ через $d_i$ ($A_i$ из границы далее будут ещё называться \textit{граничными компактами}). Вектор $d = (d_1, \ldots, d_n)$ назовём \textit{вектором решения} проблемы. Множество всех таких векторов решений для границы $A$ обозначим через $\Omega(A).$ Отметим, что разные компакты Штейнера могут задавать один и тот же элемент из $\Omega(A)$. При этом очевидно, что по элементу из $\Sigma(A)$ его вектор $d$ восстанавливается однозначно. Таким образом, множество решений проблемы Ферма--Штейнера в $\mathcal{H}(X)$ разбивается на попарно непересекающиеся классы $\Sigma_d(A),$ каждый из которых соответствует своему вектору решения $d\in \Omega(A).$ Согласно работе~\cite{ITT} в ограниченно компактных пространствах каждый класс $\Sigma_d(A)$ содержит в себе по включению единственный \textit{максимальный компакт Штейнера}, он обозначается через $K_d,$ и, вообще говоря, множество \textit{минимальных компактов Штейнера}. В~\cite{ITT} также было доказано для случая ограниченно компактных пространств, что если $d\in \Omega(A),$ то $K_d = \bigcap\limits_{i=1}^n B_{d_i}(A_i),$ где $B_{d_i}(A_i)$ --- шар (или ещё говорят замкнутая окрестность) с центром в компакте $A_i,$ см. определение в разделе~\ref{Balls}. Более того, $K\in \Sigma_d(A)$ тогда и только тогда, когда с некоторым минимальным компактом Штейнера $K_{\lambda}\in \Sigma_d(A)$ справедливо $K_{\lambda}\subset K\subset K_d.$

Геометрии пространств $\mathcal{H}(X)$ и, в частности, проблема Ферма--Штейнера в $\mathcal{H}(X)$ имеют потенциальные применения в таких прикладных областях математики, как распознавание и сравнение образов, реализация непрерывных деформаций одних геометрических объектов в другие и т.~д. Поэтому в последнее время стала активно развиваться наука по работе в различных гиперпространствах (см., например,~\cite{GeomHyp}, в которой рассматриваются кратчайшие кривые в $\mathcal{H}(X),$ или статьи~\cite{ShComp, d_GH, Is_d_GH} по исследованию и приложению более общего расстояния Громова--Хаусдорфа). 

Настоящая работа продолжает исследования, начатые в~\cite{M}, где рассматривались только границы в $\mathcal{H}(X),$ все элементы которых есть конечные подмножества $\mathbb{R}^m.$ В данной статье делается шаг в сторону перехода от исследования конечных граничных компактов к выпуклым компактным подмножествам в рамках проблемы Ферма--Штейнера. В частности, изучается возможность обобщения такого понятия, как \textit{множество сцепки}, введённого в работе~\cite{M}, см. раздел~\ref{Points}. Также исследуются вопросы того, что можно сказать при переходе от границы из конечных компактов к границе, состоящей из их выпуклых оболочек. В данной работе границы, для которых минимум функционала~(\ref{S_A}) при переходе к выпуклым оболочкам граничных компактов останется прежним, будут называться \textit{устойчивыми}, иначе --- \textit{неустойчивыми}.

В секции~\ref{Definitions} приводятся все нужные определения и вспомогательные утверждения, которые используются в статье. Раздел~\ref{Convex} разрабатывает необходимую теорию, связанную с оператором взятия выпуклой оболочки непустого компактного подмножества, в рамках проблемы Ферма--Штейнера. В частности, важным утверждением здесь является следствие~\ref{dH}. Оно используется в разделе~\ref{Saving}, который посвящён вопросам устойчивости границ. В разделе~\ref{Corresp} происходит обобщение понятия множества сцепки, которое было сделано в~\cite{M} для границ из конечных компактов. Результаты этой секции также нашли своё применение в разделе~\ref{Saving}. Секция~\ref{Example} демонстрирует приложение результатов раздела~\ref{Saving} на примере конкретной задачи из~\cite{ITT, M}.

Основными результатами статьи являются утверждение~\ref{Ass_eq}, утверждение~\ref{eq_d} и теорема~\ref{HP}, обобщающие понятие множества сцепки, теорема~\ref{pHP} о взаимосвязи границы из выпуклых компактов с $K_d$, утверждение~\ref{New_K_d}, опирающееся на следствие~\ref{dH} и дающее ответ на вопрос, что произойдёт в векторами решений из $\Omega(A)$ и каким будет максимальный компакт Штейнера в случае устойчивой границы, и теорема~\ref{DStay}, в которой приводится достаточное условие для того, чтобы граница оказалась неустойчивой.

Отметим, что утверждение~\ref{New_K_d} не является новым, его можно найти в статье~\cite{Tropin}, однако в настоящей работе приводится альтернативное доказательство этого факта.

Автор выражает благодарность своему научному руководителю, профессору А.~А.~Тужилину, и профессору А.~О.~Иванову за постановку задачи и постоянное внимание к ней в процессе совместной работы.

\section{Необходимые определения и утверждения}\label{Definitions}

В рамках данного раздела $(X, \rho)$ --- произвольное метрическое пространство, если не оговорено противное. Для удобства расстояние между двумя точками $a, b\in X$ будем обозначать через $|ab|$ вместо $\rho(a,b),$ а также вместо $(X, \rho)$ будем писать просто $X.$

Во многих местах в тексте будут использоваться следующие общепринятые обозначения:
$$[a, b) = \bigl\{(1-\lambda) a +  \lambda b \text{ } \bigl | \text{ } \lambda \in [0,1)\bigr\},$$ $$(a, b] = \bigl\{(1-\lambda) a +  \lambda b \text{ } \bigl | \text{ } \lambda \in (0,1]\bigr\},$$ $$[a, b] = \bigl\{(1-\lambda) a +  \lambda b \text{ } \bigl | \text{ } \lambda \in [0,1]\bigr\},$$ для любых двух точек $a$ и $b$ линейного метрического пространства $X$ над полем $\mathbb{R}.$

\subsection{О шарах в метрических пространствах}\label{Balls}

В определениях~\ref{one} и~\ref{two} множества $A\subset X$ и $B\subset X$ непусты.

\begin{dfn}\label{one}
\textit{Расстоянием от точки} $p\in X$ \textit{до множества} $A\subset X$ называется величина
$$|p A| = \inf\{|pa|\,:\,a\in A\}.$$
\end{dfn}

\begin{dfn}\label{two}
Множества  
$$B_r(A) = \{p\,:\,|pA|\le r\}; \,\,\, U_r(A) = \{p\,:\,|pA| < r\}$$
называются \textit{замкнутым} и \textit{открытым}, соответственно, \textit{шаром с центром в} $A$ \textit{радиуса} $r.$
\end{dfn}

В случае $A = \{a\},$ где $a\in X,$ то для краткости $B_r\bigl(\{a\}\bigr)$ и $U_r\bigl(\{a\}\bigr)$ будут заменяться на $B_r(a)$ и $U_r(a),$ соответственно.

\begin{rk}\label{rk:2}
Мы полагаем, что расстояние от точки до пустого множества равно бесконечности. В связи с этим для любого $0\le r<\infty$ верно $B_r(\emptyset) =\emptyset.$
\end{rk}

\begin{lem}\label{lm_2}
Пусть $A\subset X$ --- непустой компакт и $r>0.$ Тогда $B_r(A) = \bigcup\limits_{a\in A} B_r(a).$
\end{lem}

\begin{proof}

По определению $B_r(A) = \{x\in X \, :\, |x\, A| \le r\}.$ Покажем сначала $\{x\in X \, :\, |x\, A| \le r\}\subset \bigcup\limits_{a\in A} B_r(a).$ Пусть $p\in \{x\in X \, :\, |x\, A| \le r\}.$ Так как $A$ --- компакт, то $P_{A}(p)\neq \emptyset.$ Пусть $q\in P_{A}(p).$ Тогда $||p-q|| = |p\,A| \le r.$ Значит, $p\in B_r(q)\subset \bigcup\limits_{a\in A} B_r(a).$ В силу произвольности точки $p$ получаем $\{x\in X \, :\, |x\, A| \le r\}\subset \bigcup\limits_{a\in A} B_r(a).$

Покажем теперь $\bigcup\limits_{a\in A} B_r(a)\subset \{x\in X \, :\, |x\, A| \le r\}.$ Но для любой $p\in \bigcup\limits_{a\in A} B_r(a)$ существует точка $q\in A$ такая, что $p\in B_r(q).$ Значит, так как $q\in A,$ то $|p\, A|\le r.$ Следовательно, $p\in \{x\in X \, :\, |x\, A| \le r\}.$ Аналогично, в силу произвольности точки $p$ имеем $\bigcup\limits_{a\in A} B_r(a)\subset \{x\in X \, :\, |x\, A| \le r\}.$ 

Таким образом, $B_r(A) =\{x\in X \, :\, |x\, A| \le r\} = \bigcup\limits_{a\in A} B_r(a).$

\end{proof}

\begin{rk}\label{rk:0}
В доказательстве $\bigcup\limits_{a\in A} B_r(a)\subset \{x\in X \, :\, |x\, A| \le r\} = B_r(A)$ из леммы~\ref{lm_2} нигде не пользовалась компактность. Поэтому включение $\bigcup\limits_{a\in A} B_r(a)\subset B_r(A)$ верно для любого $A\subset X.$
\end{rk}

Следующий факт можно найти, например, в работе~\cite{M}.

\begin{lem}[\cite{M}]\label{TL}
Пусть $a_1, a_2, \ldots, a_n$ --- точки нормированного пространства со строго выпуклой нормой. Тогда если множество $$C = B_{r_1}(a_1)\cap \ldots \cap B_{r_n}(a_n)$$ состоит более чем из одной точки, то оно имеет непустую внутренность.
\end{lem}

Согласно~\cite{Mend} справедливо следующее утверждение, которое нам понадобится далее.

\begin{ass}[\cite{Mend}]\label{Mendelson}
Для непустого $A\subset X$ и функция $$f\colon X \rightarrow \mathbb{R},$$ заданая правилом $$f\colon x \mapsto |xA|,$$ непрерывна.
\end{ass}

\begin{lem}\label{bound}
Пусть $A$ --- непустое подмножество нормированного пространства $X$ и $\varepsilon\ge 0.$ Тогда $|p A| = \varepsilon$ для любой точки $p\in \partial B_{\varepsilon}(A).$
\end{lem}

\begin{proof}

Рассмотрим последовательности точек $\{x_n\}\subset B_{\varepsilon}(A)$ и $\{y_n\}\subset X\setminus B_{\varepsilon}(A),$ сходящиеся к $p.$ Заметим, что $|x_i A| \le \varepsilon$ и $|y_j A| > \varepsilon$ для всех $i,j.$ В силу утверждения~\ref{Mendelson}, а также ввиду $||p - x_i|| \rightarrow 0$ при $i \rightarrow \infty$ и $||p - y_j|| \rightarrow 0$ при $j \rightarrow \infty$ получаем $|p A| \le \varepsilon$ и $|p A|\ge \varepsilon.$ Значит, $|p A| = \varepsilon.$  

\end{proof}

Далее потребуется следующее определение.

\begin{dfn}
\textit{Суммой Минковского} двух подмножеств $A$ и $B$ линейного пространства называется множество 
$$A + B = \{a + b\,:\,a\in A, b\in B\,\}.$$ Также по определению полагаем $$\lambda A=\{\lambda a\,:\, a\in A\},$$ где $\lambda\in \mathbb{R}.$
\end{dfn}

\begin{lem}\label{sum}
Пусть $A$ --- непустое замкнутое подмножество пространства $X$ и $r, r' \ge 0.$ Тогда $B_r\bigl(B_{r'}(A)\bigr) = B_{r+r'}(A).$
\end{lem}

\begin{proof}

В силу замкнутости множества $A$ имеем $$B_r\bigl(B_{r'}(A)\bigr) = A + B_r(0) + B_{r'}(0) = A + B_{r+r'}(0) = B_{r+r'}(A).$$

\end{proof}

\begin{ass}
Пусть $A$ --- непустое выпуклое замкнутое подмножество нормированного пространства $X.$ Тогда $B_r(A)$ выпукло для любого $r\ge 0$.
\end{ass}

\begin{proof}

Из~\cite{Int_1} известно, что сумма Минковского двух выпуклых множеств выпукла. В силу замкнутости $A$ верно $B_r(0)+A=B_r(A)$. Но $B_r(0)$ и $A$ выпуклы. Поэтому $B_r(A)$ тоже выпукло.

\end{proof}

\subsection{Метрические проекции}

\begin{dfn} \label{dfn_3}
Пусть $M$ --- непустое подмножество $X.$ Множество всех подмножеств $X$ обозначим через $2^X.$ Отображение $P_M\colon X \rightarrow 2^X$, заданное по правилу $$P_M\colon x \mapsto \{z\in M : |xz| = |xM|\},$$ называется \textit{метрической проекцией} $X$ на $M.$
\end{dfn}

\begin{rk}
Если $M\subset X$ замкнуто и непусто, то для любой точки $x\in X$ множество $P_M(x)$ непусто. 
\end{rk}

В случае конечномерного нормированного пространства $X$ верен следующий факт (см., например,~\cite{C}).

\begin{ass}[\cite{C}] \label{Proj}
Пусть $M\subset X$ --- непустой выпуклый компакт. Тогда для любой точки $x\in X\setminus M$, для любой точки $y\in P_M(x)$ и для всех $\lambda \ge 0$ выполняется
$$y\in P_M\bigl((1-\lambda)y + \lambda x\bigr).$$
\end{ass}

\subsection{О расстояниях между подмножествами метрического пространства}\label{DistSec}

В определениях~\ref{distance} и~\ref{distance_H} множества $A\subset X$ и $B\subset X$ непусты.

\begin{dfn}\label{distance}
\textit{Расстоянием между} $A$ и $B$ называется величина $$|A\,B| = \inf\limits_{a\in A} |a B|.$$
\end{dfn}

\begin{dfn}\label{distance_H}
\textit{Расстоянием Хаусдорфа} между $A$ и $B$ называется величина
$$d_H(A, B) = \inf\bigl\{r\,:\, A\subset B_r(B), B\subset B_r(A)\bigr\}.$$
\end{dfn}

Геометрия расстояния Хаусдорфа довольно подробно описана, например, в работе~\cite{Schlicker}.

Обозначим множество непустых замкнутых и ограниченных подмножеств пространства $X$ через $\mathcal{H}(X).$ Известно (см., например,~\cite{Bur, IT}), что в отличие от $|A\,B|$ расстояние Хаусдорфа $d_H(A, B)$ задаёт метрику на $\mathcal{H}(X)$ . 

\begin{dfn}
Метрическое пространство $\bigl(\mathcal{H}(X), d_H\bigr)$ называется \textit{гиперпространством} над пространством $X.$
\end{dfn}

В леммах~\ref{intersection},~\ref{incl} и~\ref{aux} $X$ --- конечномерное нормированное пространство. Отметим, что лемма~\ref{intersection} используется только для доказательства леммы~\ref{incl}. 

\begin{lem}\label{intersection}
Пусть $M\in \mathcal{H}(X),$ $x\in M$ и $l$ --- луч с началом в точке $x.$ Тогда $l\cap \partial M\neq \emptyset.$
\end{lem}

\begin{proof}

Обозначим $l\cap M$ через $l'.$ Заметим, что $l'$ замкнуто как пересечение двух замкнутых множеств. Более того, $l'$ --- компакт как замкнутое подмножество компакта. Рассмотрим функцию $f\colon l' \rightarrow \mathbb{R},$ заданную правилом $a \mapsto ||a - x||.$ Функция $f$ непрерывна как функция расстояния, и на $l'$ она достигает своей точной верхней грани, так как $l'$ компакт. Пусть $p\in l'\subset M$ такая, что 
\begin{equation}\label{b}
\max\limits_{a\,\in\, l'} f(a) = f(p).
\end{equation}
Но в силу~(\ref{b}) и ввиду линейности пространства $X$ любая окрестность точки $p$ содержит точки, не лежащие в $M.$ Значит, по определению $p\in \partial M.$ Следовательно, $l\cap \partial M\neq \emptyset.$

\end{proof}

\begin{lem}\label{incl}
Пусть $A, C\in \mathcal{H}(X),$ $A$ --- выпукло, $r>0,$ $C\subset B_r(A)$ и $\bigl|C\, \partial B_r(A)\bigr| = \gamma > 0.$ Тогда для любого $0\le\delta\le\min\{r, \gamma\}$ верно $C\subset B_{r - \delta}(A).$
\end{lem}

\begin{proof}

Допустим противное, что $C\setminus B_{r - \delta}(A)\neq \emptyset.$ Пусть $x\in  C\setminus B_{r - \delta}(A)$ и $y\in P_{B_{r - \delta}(A)}(x).$ Выпустим луч $l$ из точки $y,$ проходящий через $x.$ Ввиду ограниченности множества $A$ имеем $B_r(A)\in \mathcal{H}(X).$ Отсюда по лемме~\ref{intersection} получаем $l\cap \partial B_r(A)\neq \emptyset.$ Пусть $z\in l \cap \partial B_r(A).$ По лемме~\ref{sum} справедливо $\partial B_r(A) = \partial B_{\delta}\bigl(B_{r-\delta}(A)\bigr).$ Поэтому $z \in \partial B_{\delta}\bigl(B_{r-\delta}(A)\bigr).$ При этом согласно утверждению~\ref{Proj} верно $y\in P_{B_{r - \delta}(A)}(z).$ Значит, в силу леммы~\ref{bound} $$||z-y|| = \delta.$$ Но $x\in C$, $z\in \partial B_r(A)$ и по условию $\bigl|C\, \partial B_r(A)\bigr| = \gamma,$ поэтому $$||z - x|| \ge \gamma.$$ При этом $$\delta\le\min\{r, \gamma\}\le \gamma.$$ Следовательно, $$\delta = ||z-y|| = ||z - x|| + ||x - y|| \ge \gamma + ||x - y||.$$ Отсюда $||x - y|| = 0.$ Таким образом, $x = y.$ Значит, $x\in \partial B_{r - \delta}(A)\subset B_{r - \delta}(A).$ Но по предположению $x\in  C\setminus B_{r - \delta}(A).$ Получили противоречие. Значит, $C\setminus B_{r - \delta}(A) = \emptyset,$ то есть $C\subset B_{r - \delta}(A).$ Лемма доказана.

\end{proof}

\begin{lem}\label{aux}
Пусть $M, N\in \mathcal{H}(X),$ $|M \, N| = \gamma>0$ и $0<\varepsilon<\gamma.$ Тогда $\bigl|U_{\varepsilon}(M)\,\, N\bigr| = \gamma - \varepsilon.$
\end{lem}

\begin{proof}

Покажем, что $\bigl|U_{\varepsilon}(M)\,\, N\bigr| \le \gamma - \varepsilon.$ В силу компактности $M$ и $N$ существуют такие $x\in M,$ $y\in N,$ что длина отрезка $[x, y] = \bigl\{(1-\lambda) x +  \lambda y \text{ } \bigl | \text{ } \lambda \in [0,1]\bigr\}$ равна $\gamma.$ Значит, ввиду $\varepsilon<\gamma$ можно выбрать такое $\lambda'  \in (0,1),$ что $||x - (1-\lambda') x -  \lambda' y|| = \varepsilon.$ Обозначим точку $(1-\lambda') x +  \lambda' y$ через $z.$ Значит, $|z\, M| \le \varepsilon,$ так как $x\in M.$ 

Покажем, что на самом деле $|z\, M| = \varepsilon.$ Допустим, что $|z\, M| = \varepsilon' < \varepsilon.$ Отметим, что ввиду $z\in (x,y)$ верно $||z - y|| = \gamma - \varepsilon.$ В силу компактности $M$ найдётся точка $t\in M$ такая, что $||t - z|| = \varepsilon'.$ Но тогда длина ломаной, состоящей из двух звеньев $[t,z]$ и $[z,y],$ равна $\varepsilon' + \gamma - \varepsilon < \gamma.$ Противоречие с тем, что $|M \, N| = \gamma.$ Поэтому $|z\, M| = \varepsilon.$ 

В силу линейности пространства $X$ имеем $\bigl|z\,\, U_{\varepsilon}(M)\bigr| = 0.$ Отсюда ввиду $||z - y|| = \gamma - \varepsilon$ получаем $\bigl|U_{\varepsilon}(M)\,\, N\bigr| \le \gamma - \varepsilon.$

Покажем, что $\bigl|U_{\varepsilon}(M)\,\, N\bigr| = \gamma - \varepsilon.$ Допустим, $\bigl|U_{\varepsilon}(M)\,\, N\bigr| < \gamma - \varepsilon.$ Пусть $a\in B_{\varepsilon}(M)$ и $b\in N$ такие, что $||a-b|| < \gamma - \varepsilon.$ Рассмотрим $c\in P_M(a),$ то есть $||c-a||\le \varepsilon.$ Но тогда длина ломаной, состоящей из двух звеньев $[c,a]$ и $[a,b],$ строго меньше, чем $\gamma.$ Получили противоречие с тем, что $|M \, N| = \gamma.$ Значит, $\bigl|U_{\varepsilon}(M)\,\, N\bigr| = \gamma - \varepsilon.$ Лемма доказана.

\end{proof}

\subsection{Разные вспомогательные утверждения и обозначения}

В данном подразделе $X$ --- конечномерное нормированное пространство над полем $\mathbb{R}.$

Введём отображение $\Conv\colon \mathcal{H}(X) \rightarrow \mathcal{H}(X)$,  которое каждому элементу гиперпространства $\mathcal{H}(X)$ ставит в соответствие его выпуклую оболочку.

\begin{ass}[\cite{Tropin}]\label{Lip}
Преобразование $\Conv$ является $1$-липшицевым.
\end{ass}

Согласно работе~\cite{arxiv} справедлив следующий результат.

\begin{thm}[\cite{arxiv}]\label{cont}
Пусть $A$ и $B$ --- непустые выпуклые компакты в $X.$ Тогда $$f\colon \bigl[|A\,B|, +\infty\bigr) \rightarrow \mathcal{H}(X),$$ $$f\colon r \mapsto B_r(A)\cap B,$$ непрерывна.
\end{thm}

В разделе~\ref{Saving} будет нужна следующая лемма.

\begin{lem}\label{interior}
Пусть $B$ --- выпуклое подмножество $X$ с непустой внутренностью и $U$ --- открытое подмножество $X.$ Тогда если $\partial B \cap U\neq \emptyset,$ то $\Int B \cap U\neq \emptyset.$
\end{lem}

\begin{proof}

Известно~\cite{Int_1, Int_2}, что в конечномерном нормированном пространстве для любого выпуклого множества $B$ с непустой внутренностью верно 
\begin{equation}\label{cl}
\Cl (\Int B) = \Cl B.
\end{equation}
Пусть $x\in \partial B \cap U.$ Так как $U$ --- открытое подмножество $X$ и $x\in U,$ то $U$ является окрестностью точки $x.$ Из~(\ref{cl}) получаем, что любая точка из $\Cl B$ является точкой прикосновения для $\Int B.$ Но $x\in \partial B\subset \Cl B.$ Значит, $\Int B\cap U\neq \emptyset.$ Лемма доказана.

\end{proof}

\section{Основная часть}
\markright{Основная часть}

Как отмечалось выше, в данной работе будут рассматриваться только гиперпространства $\mathcal{H}(X)$, построенные над конечномерными нормированными пространствами, а в качестве метрики на них будет браться расстояние Хаусдорфа. 

Итак, пусть $X$ всюду далее --- конечномерное нормированное пространство над полем $\mathbb{R}$, $A = \{A_1,\ldots,A_n\}\subset \mathcal{H}(X)$ --- произвольная (если не оговорено противное) граница, $d = \{d_1,\ldots,d_n\}\in \Omega(A)$ и $K_d\in \Sigma_d(A)$ --- максимальный компакт Штейнера. 

\subsection{О некоторых свойствах выпуклых оболочек}\label{Convex}

Отображение $\Conv$, вообще говоря, не сохраняет метрику Хаусдорфа $d_H$. Продемонстрируем этот факт на примере, см. рис.~\ref{1:image}. Пусть $K=\{k_1, k_2, k_3, k_4, k_5\}$ и $L=\{l_1, l_2, l_3, l_4\}$, где ближайшая точка из $L$ для $k_i$ --- это $l_i$, $|k_il_i|=c_1$ и $|k_5l_i|=c_2>c_1.$ Замечаем, что $d_H(K,L)=c_2$ и $d_H\bigl(\Conv(K),\Conv(L)\bigr)=c_1.$

\begin{figure}[h!]
	\begin{center}
		\begin{tikzpicture}
			\fill[fill=black](-1.2,-1.7) circle (2pt) node[left] {$k_1$};			
			\fill[fill=black](-0.9,-1.7) circle (1pt) node[below] {$l_1$};
			
			\fill[fill=black](-1.2,0) circle (2pt) node[left] {$k_2$};			
			\fill[fill=black](-0.9,0) circle (1pt) node[above] {$l_2$};
			
			\fill[fill=black](1.2,0) circle (2pt) node[right] {$k_3$};
			\fill[fill=black](0.9,0) circle (1pt) node[above] {$l_3$};
			
			\fill[fill=black](1.2,-1.7) circle (2pt) node[right] {$k_4$};
			\fill[fill=black](0.9,-1.7) circle (1pt) node[below] {$l_4$};
							
		         \fill[fill=black](0,-0.85) circle (2pt) node[above] {$k_5$};			
				
		\end{tikzpicture}
		\caption{Случай, когда $d_H\bigl(\Conv(K),\Conv(L)\bigr)<d_H(K,L)$.}
		\label{1:image}
	\end{center}
\end{figure}

Заметим, что в случае выше справедливо также следующее неравенство: $d_H\bigl(K,\Conv(L)\bigr)<d_H(K,L)$. Однако если взять выпуклую оболочку только одного компакта, то расстояние по Хаусдорфу может и увеличиться, см. рис.~\ref{2:image}. А именно, пусть $K=\{k_1, k_2\}$ и $L=\{l_1, l_2\}$, где $|k_il_i|=c$. Замечаем, что $d_H(K,L)=c$ и $d_H\bigl(K,\Conv(L)\bigr)=\sqrt{c^2+(\frac{|l_1l_2|}{2})^2}>c=d_H(K,L).$

\begin{figure}[h!]
	\begin{center}
		\begin{tikzpicture}
			\fill[fill=black](-1.2,-0.7) circle (1pt) node[left] {$l_1$};			
			
			\fill[fill=black](-1.2,0) circle (2pt) node[left] {$k_1$};			
			
			\fill[fill=black](1.2,0) circle (2pt) node[right] {$k_2$};
			
			\fill[fill=black](1.2,-0.7) circle (1pt) node[right] {$l_2$};
										
		\end{tikzpicture}
		\caption{Случай, когда $d_H\bigl(K,\Conv(L)\bigr)>d_H(K,L)$.}
		\label{2:image}
	\end{center}
\end{figure}

Из сказанного выше возникает вопрос: когда $\Conv$ сохранит расстояние между двумя компактами? Оказывается, справедлива следующая теорема.

\begin{thm}\label{SaveDist}
Пусть $A,B\in \mathcal{H}(X)$ и $d_H(A,B)=r$. Преобразование $\Conv$ сохраняет расстояние между $A,B$, если и только если существует $a\in A$ такая, что $U_{r}(a)\cap \Conv(B)=\emptyset,$ или существует $b\in B$ такая, что $U_{r}(b)\cap \Conv(A)=\emptyset.$
\end{thm}

\begin{proof}

Необходимость. По теореме Каратеодори любую точку $a'$ из $\Conv(A)$ можно представить в виде $a'=\sum\limits_{i=1}^{N+1} \lambda_ia_i$, где $a_i\in A$, $\lambda_i\ge0$, $\sum\limits_{i=1}^{N+1} \lambda_i=1$ и $N$ --- размерность пространства $X$. Так как $A\subset U_r\bigl(\Conv(B)\bigr)$, то для любой $a_i\in A$ существует $b_i\in \Conv(B)$ такая, что $|a_ib_i|<r=d_H(A,B)$. Замечаем, что $b'=\sum\limits_{i} \lambda_ib_i\in \Conv(B)$. Следовательно, $||a'-b'||=||\sum\limits_{i} \lambda_i(a_i-b_i)||\le \sum\limits_{i} \lambda_i||a_i-b_i||<r$. Значит, $\Conv(A)\subset U_r\bigl(\Conv(B)\bigr).$ Аналогично получаем, что $\Conv(B)\subset U_r\bigl(\Conv(A)\bigr).$ Значит, $$d_H\bigl(\Conv(A),\Conv(B)\bigr)=\min\bigl\{s : \Conv(A)\subset B_s\bigl(\Conv(B)\bigr), \Conv(B)\subset B_s\bigl(\Conv(A)\bigr)\bigr\}<r$$ --- противоречие. Поэтому существует $a\in A$ такая, что $U_{r}(a)\cap \Conv(B)=\emptyset,$ или существует $b\in B$ такая, что $U_{r}(b)\cap \Conv(A)=\emptyset.$

Достаточность. Не ограничивая общности, пусть существует $a\in A$ такая, что $U_{r}(a)\cap \Conv(B)=\emptyset.$ Допустим, что расстояние не сохранилось. Согласно утверждению~\ref{Lip} оно могло лишь уменьшиться, то есть $r>d_H\bigl(\Conv(A),\Conv(B)\bigr)$. Это означает, что для найденной точки $a\in A\subset \Conv(A)$ существует точка $b\in \Conv(B)$ такая, что $|ab|<r=d_H(A,B)$. Поэтому $b\in U_{r}(a)$. Но тогда $b\in U_{r}(a)\cap \Conv(B)$ --- противоречие.

\end{proof}

\begin{ass}\label{CapConv}
Пусть $\{K_1,\ldots , K_n\}\subset \mathcal{H}(X)$. Тогда $\Conv(\bigcap\limits_{i=1}^n K_i)\subset \bigcap\limits_{i=1}^n \Conv(K_i)$.
\end{ass}

\begin{proof}

Имеем: $\bigcap\limits_{i=1}^n K_i\subset K_j$ для всех $j$. Поэтому $\Conv(\bigcap\limits_{i=1}^n K_i)\subset \Conv(K_j)$ для всех $j$. Значит, $\Conv(\bigcap\limits_{i=1}^n K_i)\subset \bigcap\limits_{i=1}^n \Conv(K_i)$.

\end{proof}

\begin{ass}\label{BConv}
Пусть $K\in \mathcal{H}(X)$. Тогда $B_r\bigl(\Conv(K)\bigr) = \Conv\bigl(B_r(K)\bigr)$ для любого $r\ge0$.
\end{ass}

\begin{proof}

По теореме Каратеодори $$\Conv(K)=\bigcup\limits_{p_1,\ldots,p_{N+1}\in K}\bigcup\limits_{\lambda_1+\ldots+\lambda_{N+1}=1}  \lambda_1p_1+\ldots+\lambda_{N+1}p_{N+1},$$ где $\lambda_i\ge0$ и $N$ --- размерность пространства $X$. Также $$\bigcup\limits_{p_1,\ldots,p_{N+1}\in K}\bigcup\limits_{\lambda_1+\ldots+\lambda_{N+1}=1}  \lambda_1p_1+\ldots+\lambda_{N+1}p_{N+1} = \bigcup\limits_{\lambda_1+\ldots+\lambda_{N+1}=1}  \lambda_1K+\ldots+\lambda_{N+1}K.$$ Отсюда имеем в силу свойств суммы Минковского
\begin{multline*}
\\ B_r\bigl(\Conv(K)\bigr) = \Conv(K) + B_r(0) = \\
= \Biggl(\bigcup\limits_{\lambda_1+\ldots+\lambda_{N+1}=1}  \lambda_1K+\ldots+\lambda_{N+1}K\Biggr) + B_r(0) = \bigcup\limits_{\lambda_1+\ldots+\lambda_{N+1}=1}  \Bigl(\lambda_1K+\ldots+\lambda_{N+1}K + B_r(0)\Bigr) = \\
= \bigcup\limits_{\lambda_1+\ldots+\lambda_{N+1}=1}  \lambda_1\bigl(K+B_r(0)\bigr)+\ldots+\lambda_{N+1}\bigl(K+B_r(0)\bigr) = \Conv\bigl(K+B_r(0)\bigr) = \Conv\bigl(B_r(K)\bigr). \\
\end{multline*}
\end{proof}

Обозначим $\bigcap\limits_{i=1}^n B_{d_i}\bigl(\Conv(A_i)\bigr)$ через $K_d^{\Conv}$.

\begin{cor}\label{KdConv}
$\Conv(K_d)\subset  K_d^{\Conv}$.
\end{cor}

\begin{proof}

По утверждению~\ref{CapConv} имеем $$\Conv(K_d)=\Conv\bigl(\bigcap\limits_{i=1}^n B_{d_i}(A_i)\bigr)\subset \bigcap\limits_{i=1}^n \Conv\bigl(B_{d_i}(A_i)\bigr).$$ 

По утверждению~\ref{BConv} получаем $$\bigcap\limits_{i=1}^n \Conv\bigl(B_{d_i}(A_i)\bigr) = \bigcap\limits_{i=1}^n B_{d_i} \bigl(\Conv(A_i)\bigr)=K_d^{\Conv}.$$

Поэтому $\Conv(K_d)\subset  K_d^{\Conv}$.

\end{proof}

\begin{cor}\label{dH}
Для всех $i$ выполняется $d_H\bigl(\Conv(A_i),K_d^{\Conv}\bigr)\le d_i$.
\end{cor}

\begin{proof}

Имеем $K_d^{\Conv}\subset B_{d_i}\bigl(\Conv(A_i)\bigr)$. С другой стороны, так как $A_i \subset B_{d_i}(K_d),$ то $$\Conv(A_i)\subset \Conv\bigl(B_{d_i}(K_d)\bigr).$$ По утверждению~\ref{BConv} получаем $$\Conv\bigl(B_{d_i}(K_d)\bigr) = B_{d_i}\bigl(\Conv(K_d)\bigr).$$ И в силу следствия~\ref{KdConv} $$B_{d_i}\bigl(\Conv(K_d)\bigr)\subset B_{d_i}\bigl(K_d^{\Conv}\bigr).$$ Таким образом, $\Conv(A_i)\subset B_{d_i}\bigl(K_d^{\Conv}\bigr)$ для всех $i$. Следовательно, $$d_H\bigl(\Conv(A_i),K_d^{\Conv}\bigr)\le d_i.$$

\end{proof}

\subsection{Далёкие, неплотные и дискретные точки, их взаимосвязь}\label{Points}

Введём следующее определение.

\begin{dfn}\label{far_dfn}
Точку $a\in A_i$ назовём \textit{далёкой}, если $U_{d_i}(a)\cap K_d = \emptyset.$
\end{dfn}

\begin{ass}\label{Part}
Если $a\in A_i$ --- далёкая и $d_i > 0$, то $B_{d_i}(a)\cap K_d \subset \partial K_d$.
\end{ass}

\begin{proof}
Пусть $a\in A_i$ --- далёкая, то есть $U_{d_i}(a)\cap K_d = \emptyset.$ Значит, $B_{d_i}(a)\cap K_d \subset \partial B_{d_i}(a).$ В нормированном пространстве любая точка из $\partial B_{d_i}(a)$ является точкой прикосновения для $U_{d_i}(a).$ Но $B_{d_i}(a)\cap K_d\subset K_d.$ Поэтому, так как $U_{d_i}(a)\cap K_d = \emptyset$ и $d_i > 0,$ в любой окрестности точки из $B_{d_i}(a)\cap K_d$ содержатся как точки из $K_d,$ так и точки, не лежащие в $K_d.$ Следовательно, по определению граничных точек $B_{d_i}(a)\cap K_d \subset \partial K_d$.

\end{proof}

\begin{dfn}
Границу $A$, все элементы которой являются выпуклыми компактами, назовём \textit{выпуклой}.
\end{dfn}

\begin{ass}\label{F_part}
Пусть $A = \{A_1,\ldots,A_n\}$ --- выпуклая граница, $a\in A_i$ --- далёкая точка и $d_i > 0$. Тогда $a\in \partial A_i$.
\end{ass}

\begin{proof} 

Так как $a\in A_i$ --- далёкая, то $U_{d_i}(a)\cap K_d = \emptyset.$ Следовательно, так как $d_i > 0,$ то $a\in X\setminus K_d.$ Компакт $K_d$ выпуклый в силу выпуклости границы. Поэтому согласно утверждению~\ref{Proj} для точки $a$ и для любой точки $p\in P_{K_d}(a)$ верно $p\in P_{K_d}(a_{\lambda})$ при всех $\lambda\ge 0,$ где $a_{\lambda} = (1 - \lambda)p + \lambda a.$

Заметим, что $B_{d_i}(a)\cap K_d \neq \emptyset,$ так как $A_i \subset B_{d_i}(K_d)$. Значит, $|a K_d| = d_i$ ввиду того, что $a$ --- далёкая точка. Но $p\in P_{K_d}(a)$, поэтому $|a p| = d_i$ по определению~\ref{dfn_3}.

Допустим, что $a\in \Int A_i.$ Тогда в силу линейности пространства $X$ существует такое $\lambda > 1$, что $|a_{\lambda}p| > |a p| = d_i$ и при этом $a_{\lambda}\in A_i.$  Ввиду того, что $p\in P_{K_d}(a_{\lambda})$, имеем $|a_{\lambda}K_d| > d_i.$ Значит, $A_i\not \subset B_{d_i}(K_d).$ Получили противоречие с тем, что $d_i = d_H(A_i,K_d).$ Поэтому $a\in \partial A_i$.

\end{proof}

Вообще говоря, неясно, для любой ли конечной границы $A\subset \mathcal{H}(X)$ далёкие точки всегда найдутся хотя бы для одного вектора $d\in \Omega(A)$. Например, возможно, не исключена ситуация, изображённая на рисунке~\ref{Example_1}. А именно, $A = \{\mathfrak{A}, \mathfrak{B}, \mathfrak{C}\}$, где $\mathfrak{A}=\{a_1,a_2,a_3\}; \mathfrak{B}=\{b_1,b_2,b_3,b_4\}; \mathfrak{C}=\{c_1,c_2,c_3,c_4\}$. Предположим, что вектор $$d = \Bigl(d_1=d_H(\mathfrak{A},K_d); d_2=d_H(\mathfrak{B},K_d); d_3=d_H(\mathfrak{C},K_d)\Bigr),$$ по которому построен 
\begin{multline*}
K_d = \Bigl(B_{d_1}(a_1)\cap B_{d_2}(b_1)\cap B_{d_3}(c_1)\Bigr)\cup \\ \cup \Bigl(B_{d_1}(a_2)\cap B_{d_3}(c_1)\Bigr)\cup \Bigl(B_{d_1}(a_2)\cap B_{d_2}(b_3)\cap B_{d_3}(c_2)\Bigr) \cup \Bigl(B_{d_1}(a_3)\cap B_{d_2}(b_3)\Bigr)\cup \\ \cup \Bigl(B_{d_1}(a_3)\cap B_{d_2}(b_4)\cap B_{d_3}(c_4)\Bigr),
\end{multline*} 
принадлежит $\Omega(A)$. Причём, оба множества $$B_{d_1}(a_2)\cap B_{d_3}(c_1)\subset U_{d_2}(b_2)$$ и $$B_{d_1}(a_3)\cap B_{d_2}(b_3)\subset U_{d_3}(c_3)$$ одноточечны. Замечаем, что если мы сколь угодно уменьшим $d_1$, то перестроенный в соответствии с уменьшенным расстоянием $d_1$ компакт $K_d,$ обозначим его через $K_{d'}$ уже не будет компактом Штейнера, так как тогда, например, $b_2\not \in B_{d_2}(K_{d'})$, чего быть не может по определению расстояния Хаусдорфа. Аналогично, если сколь угодно уменьшим $d_2$, то $c_3\not \in B_{d_3}(K_{d'})$ и если сколь угодно уменьшим $d_3$, то снова $b_2\not \in B_{d_2}(K_{d'}).$

При этом $U_{d_1}(a_i)\cap K_d\neq \emptyset$, $U_{d_2}(b_i)\cap K_d\neq \emptyset$ и $U_{d_3}(c_i)\cap K_d\neq \emptyset$ для всех $i.$ Таким образом, если в самом деле данный вектор $d$ является решением проблемы Ферма--Штейнера для границы $A$, то у всех граничных компактов из $A$ нет далёких точек.

\begin{figure}[h!]
	\begin{center}
		\begin{tikzpicture}
		\draw[line width=1mm] (-3.3,0) circle(1);
		\fill[fill=black](-3.3,0) circle (1pt) node[below] {$a_1$};
		
		\draw[line width=0.5mm] (-2.7,1) circle(0.9);
		\fill[fill=black](-2.7,1) circle (1pt) node[above] {$b_1$};
		
		\draw[thin] (-1.85,0) circle(0.8);
		\fill[fill=black](-1.85,0) circle (1pt) node[above] {$c_1$};
		
		\draw[line width=1mm] (0,0) circle(1);
		\fill[fill=black](0,0) circle (1pt) node[below] {$a_2$};		
		
		\draw[line width=0.5mm] (-1.05,-0.6) circle(0.9);
		\fill[fill=black](-1.05,-0.6) circle (1pt) node[below] {$b_2$};
		
		\draw[line width=0.5mm] (1.5,0) circle(0.9);
		\fill[fill=black](1.5,0) circle (1pt) node[above] {$b_3$};

		\draw[thin] (0.55,1) circle(0.8);
		\fill[fill=black](0.55,1) circle (1pt) node[above] {$c_2$};
		
		\draw[line width=1mm] (3.45,0) circle(1);
		\fill[fill=black](3.45,0) circle (1pt) node[below] {$a_3$};
		
		\draw[thin] (2.4,-0.5) circle(0.8);
		\fill[fill=black](2.4,-0.5) circle (1pt) node[below] {$c_3$};
		
		\draw[thin] (4.9,0) circle(0.8);
		\fill[fill=black](4.9,0) circle (1pt) node[above] {$c_4$};
		
		\draw[line width=0.5mm] (4,1) circle(0.9);
		\fill[fill=black](4,1) circle (1pt) node[above] {$b_4$};

		\fill[fill=black](2.43,0) circle (2.5pt) node[below] {};
		\fill[fill=black](-1,0) circle (2.5pt) node[below] {};
		
		\draw[thick, dotted] (4.25,0.3) to (0.5,-3) node[below] {$K_d$};
		\draw[thick, dotted] (2.5,0) to (0.5,-3) node[below] {};
		\draw[thick, dotted] (0.8,0.4) to (0.5,-3) node[below] {};
		\draw[thick, dotted] (-1,0) to (0.5,-3) node[below] {};
		\draw[thick, dotted] (-2.48,0.3) to (0.5,-3) node[below] {};
		\end{tikzpicture}
		\caption{Случай границы без далёких точек.}
		\label{Example_1}
	\end{center} 
\end{figure}

\begin{dfn}
Точку $a\in A_i$ назовём \textit{неплотной}, если $\Int \bigl(B_{d_i}(a)\cap K_d\bigr) = \emptyset.$
\end{dfn}

\begin{rk}\label{eq}
В силу определения~\ref{far_dfn} далёкие точки всегда являются неплотными. Действительно, если $a$ --- далёкая, то $U_{d_i}(a)\cap K_d = \emptyset$ и поэтому $B_{d_i}(a)\cap K_d\subset \partial B_{d_i}(a)$ при $d_i > 0,$ либо $B_{d_i}(a)\cap K_d = \{a\}$ при $d_i = 0.$ Следовательно, $\Int \bigl(B_{d_i}(a)\cap K_d\bigr) = \emptyset,$ то есть $a$ --- неплотная точка. 
\end{rk}

Однако обратное, вообще говоря, неизвестно. А именно, неясно, всегда ли неплотные точки являются далёкими.

Замыкание множества $M\subset X$ всюду далее будем обозначать через $\Cl M.$

\begin{ass}\label{Ass_eq}
Если $\Cl (\Int K_d) = K_d,$ то все неплотные точки совпадают с далёкими.
\end{ass}

\begin{proof}

Пусть $F(A)$ --- это множество всех далёких точек границы $A$, а $L(A)$ --- это множество всех неплотных точек для $A.$ Согласно замечанию~\ref{eq} имеем $F(A)\subset L(A).$ Поэтому нужно доказать $L(A)\subset F(A)$ при условии, что $\Cl (\Int K_d) = K_d.$

Пусть $a\in A_i$ --- неплотная точка, то есть $\Int \bigl(B_{d_i}(a)\cap K_d\bigr) = \emptyset.$ Покажем, что тогда $a\in F(A)$, то есть $U_{d_i}(a)\cap K_d = \emptyset.$ Допустим противное, а именно, $U_{d_i}(a)\cap K_d \neq \emptyset.$ Возьмём $x\in U_{d_i}(a)\cap K_d.$ В силу условия $\Cl (\Int K_d) = K_d$ любая точка из $K_d$ является точкой прикосновения для $\Int K_d.$ Значит, так как $U_{d_i}(a)$ является окрестностью взятой точки $x\in K_d,$ то $U_{d_i}(a)\cap \Int K_d\neq \emptyset.$ Хорошо известно, что в топологических пространствах внутренность пересечения конечного числа множеств равна пересечению их внутренностей. Поэтому $\emptyset \neq U_{d_i}(a)\cap \Int K_d = \Int B_{d_i}(a)\cap \Int K_d = \Int \bigl(B_{d_i}(a)\cap K_d\bigr) = \emptyset.$ Получили противоречие. 

Значит, из $\Int \bigl(B_{d_i}(a)\cap K_d\bigr) = \emptyset$ следует $U_{d_i}(a)\cap K_d = \emptyset.$ Поэтому $L(A)\subset F(A).$ Отсюда $L(A) = F(A).$ Утверждение доказано.

\end{proof}

\begin{cor}\label{Cor_eq}
Если $\Int K_d\neq \emptyset$ и граница $A$ выпукла, то все неплотные точки совпадают с далёкими.
\end{cor}

\begin{proof}

Граница $A$ выпукла, поэтому $K_d$ --- выпуклый компакт как пересечение выпуклых множеств.

Известно~\cite{Int_1, Int_2}, что в конечномерном нормированном пространстве для любого выпуклого множества $M$ с непустой внутренностью верно $\Cl (\Int M) = \Cl M.$ Так как $K_d$ --- компакт, то $\Cl K_d = K_d.$ Следовательно, $\Cl (\Int K_d) = K_d.$

Таким образом, мы приходим к условию утверждения~\ref{Ass_eq}. Значит, в границе $A$ все неплотные точки совпадают с далёкими.

\end{proof}

\begin{thm}\label{HP}
Пусть граница $A$ является выпуклой. Тогда для любого класса $\Sigma_d(A)$ существует граничный компакт $A_i,$ содержащий далёкую точку.
\end{thm}

\begin{proof}

Допустим противное, то есть что для всех $A_i$ и всех точек $a\in A_i$ выполнено $U_{d_i}(a)\cap K_d\neq \emptyset$. Но тогда $A_i \subset U_{d_i}(K_d)$ для любого $i.$ Значит, положительна величина:
$$\gamma_i = \bigl|A_i \, \partial B_{d_i}(K_d)\bigr|.$$ Следовательно, в силу конечного числа компактов в границе $A$ получаем 
$$\gamma = \min\limits_{i\in [1, n]} \gamma_i > 0.$$
Таким образом, для любого $i$ согласно лемме~\ref{incl} имеем $$A_i \subset B_{d_i-\gamma}(K_d).$$

По условию все $A_i$ выпуклы, значит, $K_d$ тоже выпуклый. Поэтому согласно теореме~\ref{cont} множество $B_{d_i}(A_i)\cap K_d$ меняется непрерывно при малых возмущениях $d_i$. Зафиксируем номер $i.$ Таким образом, для $0<\varepsilon<\gamma$ существует $\delta>0$ такое, что $$K_d = B_{d_i}(A_i)\cap K_d \subset B_{\varepsilon}\bigl(B_{d_i-\delta}(A_i)\cap K_d\bigr).$$ Обозначим $B_{d_i-\delta}(A_i)\cap K_d$ через $K.$ Таким образом, $K_d\subset B_{\varepsilon}(K).$ Значит, для любого $j$ согласно лемме~\ref{sum} верно $$A_j \subset B_{d_j-\gamma}(K_d) \subset B_{d_j-\gamma}\bigl(B_{\varepsilon}(K)\bigr) = B_{d_j-\gamma+\varepsilon}(K).$$ При этом $K\subset K_d\subset B_{d_j}(A_j)$ для всех $j$ и $K\subset B_{d_i-\delta}(A_i).$ Значит, $d_H(K, A_j)\le d_j$ и $d_H(K, A_i)\le \max(d_i-\delta, d_i-\gamma+\varepsilon) < d_i.$ Следовательно, $S(A, K) < S(A, K_d),$ получили противоречие. Теорема доказана.

\end{proof}

\begin{cor}
Пусть граница $A$ является выпуклой. Тогда для любого класса $\Sigma_d(A)$ существует граничный компакт $A_i,$ содержащий неплотную точку. 
\end{cor}

\begin{proof}
Согласно следствию~\ref{Cor_eq} в случае выпуклой границы далёкие точки совпадают с неплотными. Но в силу теоремы~\ref{HP} у выпуклой границы существует по крайней мере одна далёкая точка. Значит, существует по крайней мере одна неплотная точка.
\end{proof}

\begin{dfn}
Границу $A$, все элементы которой являются конечными множествами, назовём \textit{финитной}.
\end{dfn}

\begin{dfn}
Точку $a\in A_i$ назовём \textit{дискретной}, если $\# B_{d_i}(a)\cap K_d < \infty.$
\end{dfn}

\begin{ass}\label{MS_1}
Если $a\in A_i$ дискретна, то $B_{d_i}(a)\cap K_d \subset \partial K_d.$
\end{ass}

Доказательство утверждения~\ref{MS_1} дословно повторяет соответствующее доказательство из~\cite{M} для случая $X = \mathbb{R}^m$ и дискретной границы $A$ с попарно непересекающимися компактами.

\begin{thm}\label{MS_2}
Пусть $A\subset \mathcal{H}(X)$ --- финитная граница и норма пространства $X$ строго выпукла. Тогда для любого класса $\Sigma_d(A)$ по крайней мере в одном граничном компакте $A_i$ найдётся дискретная точка.
\end{thm}

Доказательство теоремы~\ref{MS_2} в точности такое же, как и в~\cite{M} для случая $X = \mathbb{R}^m$ и попарно непересекающимися граничными компактами. 

\begin{ass}\label{eq_d}
В случае финитной границы $A\subset \mathcal{H}(X)$ и пространства $X$ со строго выпуклой нормой неплотные точки совпадают с дискретными.
\end{ass}

\begin{proof}

Пусть $D(A)$ --- это множество всех дискретных точек границы $A$, а $L(A)$ --- это множество всех неплотных точек для $A.$ Согласно утверждению~\ref{MS_1} имеем $D(A)\subset L(A).$ Поэтому нужно доказать, что $L(A)\subset D(A)$ в случае финитной границы $A$ и строго выпуклой нормы $X$.

Пусть $a\in A_i$ --- неплотная точка, то есть $\Int \bigl(B_{d_i}(a)\cap K_d\bigr) = \emptyset.$ Покажем, что тогда $a\in D(A)$, то есть $\# B_{d_i}(a)\cap K_d < \infty.$ 

Имеем $$K_d = \bigcup\limits_{j_1,\ldots,j_n} B_{d_1}(a_{j_1}^1)\cap \ldots \cap B_{d_n}(a_{j_n}^n).$$ Так как по условию $\Int \bigl(B_{d_i}(a)\cap K_d\bigr) = \emptyset,$ то $$\Int \bigl(B_{d_i}(a)\cap B_{d_1}(a_{j_1}^1)\cap \ldots \cap B_{d_n}(a_{j_n}^n)\bigr) = \emptyset$$ для всех наборов $j_1,\ldots,j_n.$ Значит, согласно лемме~\ref{TL} каждое множество $$B_{d_i}(a)\cap B_{d_1}(a_{j_1}^1)\cap \ldots \cap B_{d_n}(a_{j_n}^n)$$ одноточечно. В силу финитности границы $A$ таких множеств конечное число. Поэтому множество $B_{d_i}(a)\cap K_d$ состоит из конечного числа точек. Следовательно, точка $a$ является дискретной точкой. Отсюда в силу произвольности $a$ получаем $L(A)\subset D(A).$ Значит, $L(A) = D(A).$ Утверждение доказано.

\end{proof}

\begin{cor}
Пусть граница $A$ является финитной, а пространство $X$ имеет строго выпуклую норму. Тогда для любого класса $\Sigma_d(A)$ существует граничный компакт $A_i,$ содержащий неплотную точку. 
\end{cor}

\begin{proof}

Согласно утверждению~\ref{eq_d} в случае финитной границы и пространства $X$ со строго выпуклой нормой дискретные точки совпадают с неплотными. Но в силу теоремы~\ref{MS_2} у финитной границы существует по крайней мере одна дискретная точка. Значит, существует по крайней мере одна неплотная точка.

\end{proof}

\subsection{О взаимосвязи выпуклой границы с максимальным компактом Штейнера}\label{Corresp}

Введём ряд обозначений:
\begin{itemize}
	\item $F_i^A$ --- множество всех далёких точек компакта $A_i\in A$ и $F^A = \bigcup\limits_{i} F_i^A;$
	\item $\HP_d(p, F_i^A) = B_{d_i}(p)\cap K_d,$ где $p\in F_i^A;$
	\item $\HP_d(F_i^A) = \bigcup\limits_{p\in F_i^A} \HP_d(p, F_i^A);$
	\item $\HP_d(F^A) = \bigcup\limits_{i} \HP_d(F_i^A);$
\end{itemize}

Везде, где будет ясно, о какой границе $A$ идёт речь, верхний индекс $A$ будет опускаться, то есть $F_i^A$ и $F^A$ будут заменяться на $F_i$ и $F$ соответственно.

Леммы~\ref{zero_dist} и~\ref{subset_K_d} сформулированы, чтобы немного подробнее раскрыть геометрию множества $\HP_d(F_i),$ которое далее будет часто использоваться.

\begin{lem}\label{zero_dist}
Если $d_i=0,$ то $\HP_d(F_i) = A_i.$
\end{lem}

\begin{proof}

Так как $d_i=0,$ то для любой точки $a\in A_i$ верно $U_{d_i}(a) = \emptyset$, и значит, $U_{d_i}(a)\cap K_d = \emptyset.$ Следовательно, $F_i = A_i.$ Ввиду того, что $0=d_i=d_H(A_i,K_d),$ имеем $A_i = K_d.$ Поэтому $\HP_d(F_i) = B_{d_i}(F_i)\cap K_d = A_i.$

\end{proof}

\begin{lem}\label{subset_K_d}
Если $A_i\subset K_d$ и $A_i\neq K_d$, то $\HP_d(F_i) = \emptyset.$
\end{lem}

\begin{proof}

Так как $A_i\subset K_d$ и $A_i\neq K_d,$ то $d_i > 0.$ Значит, для любой точки $a\in A_i$ справедливо $U_{d_i}(a)\neq \emptyset.$ При этом $a\in A_i\subset K_d.$ Следовательно, $U_{d_i}(a)\cap K_d\neq \emptyset$ для любой точки $a\in A_i.$ Поэтому $F_i = \emptyset.$ Отсюда согласно замечанию~\ref{rk:2} верно $\HP_d(F_i) = B_{d_i}(F_i)\cap K_d = \emptyset.$

\end{proof}

Главным результатом, содержащимся в данном подразделе, является теорема о взаимосвязи выпуклой границы с $K_d$~\ref{pHP}. Однако прежде чем переходить к её формулировке и доказательству, требуется ввести некоторые вспомогательные леммы и замечания, касающиеся множества $\HP_d(F_i).$

Заметим, что $\HP_d(F_i)$ можно определить следующим образом. Ввиду того, что $F_i$ --- это множество всех далёких точек в $A_i,$ имеем $F_i = A_i\setminus U_{d_i}(K_d).$ Тогда по лемме~\ref{lm_2} в силу компактности $F_i$ верно $$\HP_d(F_i) = \bigcup\limits_{p\in F_i} \HP_d(p, F_i) = \bigcup\limits_{p\in F_i}\bigl(B_{d_i}(p)\cap K_d\bigr) = \bigl(\bigcup\limits_{p\in F_i}B_{d_i}(p)\bigr)\cap K_d = B_{d_i}(F_i)\cap K_d.$$

\begin{rk}\label{lm_0}
Множество $\HP_d(F_i)$ --- компакт как пересечение двух компактов.
\end{rk}

\begin{lem}\label{empt}
Множество $\HP_d(F_i)$ непусто тогда и только тогда, когда $F_i$ непусто.
\end{lem}

\begin{proof}

Согласно лемме~\ref{lm_2} $$\HP_d(F_i)=B_{d_i}(F_i)\cap K_d = \bigcup\limits_{f\in F_i\subset A_i}B_{d_i}(f)\cap K_d.$$ Напомним, что для любой точки $a\in A_i$ верно $B_{d_i}(a)\cap K_d\neq\emptyset.$ Следовательно, $\bigcup\limits_{f\in F_i\subset A_i}B_{d_i}(f)\cap K_d\neq \emptyset$ тогда и только тогда, когда $F_i\neq \emptyset.$ Лемма доказана.

\end{proof}

Пусть граница $A$ выпукла, $\HP_d(F_i)\neq \emptyset$ и для некоторого $j\neq i$ верно $\HP_d(F_i)\subset U_{d_j}(A_j).$ Положим 
\begin{equation}\label{HPF_dist}
\gamma=\bigl|\HP_d(F_i) \,\, \partial B_{d_j}(A_j)\bigr|.
\end{equation}
Тогда справедлива следующая лемма.

\begin{lem}\label{lm_1}
Для любого $0<\varepsilon<\gamma$ найдётся $\Delta>0$ такое, что для любого $0\le\delta\le \Delta$ при $$H = B_{d_i+\delta}(F_i)\cap K_d$$ верно $$\bigl|H \, \partial B_{d_j}(A_j)\bigr| \ge \gamma - \varepsilon > 0.$$ 
\end{lem}

\begin{proof}

Пусть $p\in F_i$, то есть $U_{d_i}(p)\cap K_d = \emptyset.$ По условию все граничные компакты выпуклы, значит, $K_d$ тоже выпуклый. Поэтому согласно теореме~\ref{cont} множество $\HP_d(p,F_i) = B_{d_i}(p)\cap K_d$ меняется непрерывно при увеличении $d_i.$ 

Пусть $0<\varepsilon<\gamma.$ Тогда для $\varepsilon$ найдётся такое $\Delta>0,$ что для любого $0\le\delta\le \Delta$ верно $$H(p) := B_{d_i+\delta}(p)\cap K_d\subset B_{d_i+\Delta}(p)\cap K_d\subset U_{\varepsilon}\bigl(B_{d_i}(p)\cap K_d\bigr) = U_{\varepsilon}\bigl(\HP_d(p,F_i)\bigr).$$ Заметим, что $\HP_d(p,F_i)\subset \HP_d(F_i).$ Напомним, что согласно~(\ref{HPF_dist}) имеем $\bigl|\HP_d(F_i) \,\, \partial B_{d_j}(A_j)\bigr| = \gamma.$ Поэтому $\bigl|\HP_d(p,F_i) \,\, \partial B_{d_j}(A_j)\bigr| =: \gamma' \ge \gamma.$ Но тогда по лемме~\ref{aux} $$\Bigl|U_{\varepsilon}\bigl(\HP_d(p,F_i)\bigr)\,\, \partial B_{d_j}(A_j)\Bigr| = \gamma' - \varepsilon \ge \gamma - \varepsilon.$$ Значит, так как $H(p) \subset U_{\varepsilon}\bigl(\HP_d(p,F_i)\bigr),$ то $$\bigl|H(p) \, \partial B_{d_j}(A_j)\bigr| \ge \gamma - \varepsilon  > 0.$$ Отсюда $$\inf\limits_{p\in F_i} \bigl|H(p) \, \partial B_{d_j}(A_j)\bigr|\ge \gamma - \varepsilon  > 0.$$ В силу компактности $F_i,$ по лемме~\ref{lm_2} имеем равенство $$\bigcup\limits_{p\,\in F_i} H(p) = \bigcup\limits_{p\,\in F_i} \bigl(B_{d_i + \delta}(p)\cap K_d\bigr) = \bigl(\bigcup\limits_{p\,\in F_i} B_{d_i + \delta}(p)\bigr)\cap K_d = B_{d_i + \delta}(F_i)\cap K_d = H.$$ Следовательно, $\bigl|H \, \partial B_{d_j}(A_j)\bigr| \ge \gamma - \varepsilon > 0.$ Лемма доказана.

\end{proof}

\begin{thm}[О взаимосвязи выпуклой границы с $K_d$]\label{pHP}
Пусть граница $A$ выпукла и все $d_i$ положительны. Тогда для любого номера $i$ существует точка $p\in \HP_d(F)$ такая, что $p\in \HP_d(F_i)$ или $p\in \partial B_{d_i}(A_i)$.
\end{thm}

\begin{proof}
Допустим противное, а именно, пусть существует номер $i$ такой, что для любой $p\in \HP_d(F)$ верно $p\notin \HP_d(F_i)$ и $p\notin \partial B_{d_i}(A_i)$. Значит, $$\HP_d(F_i)=\emptyset.$$ Также по определению $\HP_d(F)\subset K_d.$ Но $K_d = \bigcap\limits_{i=1}^n B_{d_i}(A_i).$ Значит, $\HP_d(F)\subset B_{d_i}(A_i).$ Следовательно, ввиду $\HP_d(F)\cap \partial B_{d_i}(A_i) = \emptyset$ имеем $$\HP_d(F)\subset U_{d_i}(A_i).$$

Так как $\HP_d(F_j)\subset U_{d_i}(A_i)$ и $\HP_d(F_j)$ --- компакт согласно замечанию~\ref{lm_0}, то $\gamma_j = \bigl|\HP_d(F_j) \, \partial B_{d_i}(A_i)\bigr| > 0$ для всех непустых $\HP_d(F_j).$ По теореме~\ref{HP} существует по крайней мере один $A_j\in A$ такой, что $\HP_d(F_j)\neq \emptyset.$ Поэтому определено $$\gamma = \min\limits_{j: \, \HP_d(F_j) \, \neq \, \emptyset} \gamma_j > 0.$$ 

Пусть $\HP_d(F_k)\neq \emptyset,$ что эквивалентно $F_k\neq \emptyset$ по лемме~\ref{empt}. Тогда согласно лемме~\ref{lm_1} для $0<\varepsilon< \gamma$ найдётся $R_k>0$ такое, что для любого $0<r_k\le R_k$ верно $\Bigl|\bigl(B_{d_k+r_k}(F_k)\cap K_d\bigr) \,\, \partial B_{d_i}(A_i)\Bigr| \ge \gamma_k - \varepsilon \ge \gamma - \varepsilon > 0.$ 

Введём обозначения:
$$r = \min\limits_{j: \, \HP_d(F_j)\,\neq \, \emptyset} r_j>0;$$  
\begin{equation}\label{first}
H_j = B_{d_j+r}(F_j)\cap K_d.
\end{equation}
Согласно замечанию~\ref{rk:2} если $F_j=\emptyset,$ то $B_{d_j+r}(F_j)= \emptyset$ и поэтому $H_j= \emptyset.$ С другой стороны, так как $F_j\subset A_j$ и $B_{d_j}(a)\cap K_d\neq \emptyset$ для любой $a\in A_j,$ то из $H_j = \emptyset$ следует $B_{d_j+r}(F_j) = \emptyset.$ Но $F_j\subset B_{d_j+r}(F_j).$ Значит, $F_j = \emptyset.$ Таким образом, $H_j = \emptyset$ тогда и только тогда, когда $F_j = \emptyset.$

Заметим, что из сказанного выше вытекает
\begin{equation}\label{second}
\eta := \min\limits_{j: \, H_j\,\neq \, \emptyset} \bigl|H_j \, \partial B_{d_i}(A_i)\bigr| = \min\limits_{j: \, F_j\,\neq \, \emptyset} \bigl|H_j \, \partial B_{d_i}(A_i)\bigr|\ge \gamma - \varepsilon > 0.
\end{equation}
Отметим, что согласно замечанию~\ref{rk:2} при $F_j = \emptyset$ верно $U_{r}(F_j)\subset B_{r}(F_j) = \emptyset.$ Тогда для
\begin{equation}\label{second_00}
A'_j := A_j\setminus U_{r}(F_j)
\end{equation}
имеем $A'_j\subset A_j$ и при $F_j = \emptyset$ получаем $A'_j = A_j.$

Замечаем, что $A'_j \subset U_{d_j}(K_d)$ для любого $j.$ Следовательно, в силу компактности $A'_j,$ для всех непустых $A'_j$ положительна величина $\mu_j = \bigl|A'_j \, \partial B_{d_j}(K_d)\bigr|.$ Поэтому, так как все $d_i$ положительны, получаем 
$$\mu = \min\{\min\limits_{A'_j\,\neq\, \emptyset} \mu_j, \min\limits_{i} d_i\} > 0.$$
Таким образом, для любого $j$ в силу леммы~\ref{incl} имеем 
\begin{equation}\label{second_0}
A'_j \subset B_{d_j-\mu}(K_d).
\end{equation}

По условию все $A_j$ выпуклы. Также по условию $\HP_d(F_i) = \emptyset.$ Значит, для любой точки $a\in A_i$ верно $U_{d_i}(a)\cap K_d\neq \emptyset.$ Следовательно, верно $|A_i \, K_d| < d_i.$ Поэтому согласно теореме~\ref{cont} множество $B_{d_i}(A_i)\cap K_d$ меняется непрерывно при небольших уменьшениях $d_i$. Значит, для $0<\varepsilon<\min\{\eta, \mu\}$ существует $0<\delta\le d_i - |A_i \, K_d|$ такое, что 
$$K_d = B_{d_i}(A_i)\cap K_d \subset B_{\varepsilon}\bigl(B_{d_i - \delta}(A_i)\cap K_d\bigr).$$
Без ограничения общности будем считать $\delta \le \varepsilon.$ Для удобства введём обозначение 
\begin{equation}\label{third_0}
K=B_{d_i - \delta}(A_i)\cap K_d.
\end{equation}
Таким образом, 
\begin{equation}\label{third_00}
K_d\subset B_{\varepsilon}(K).
\end{equation}
Для любого $j$ в силу~(\ref{first}) верно $H_j\subset K_d\subset B_{d_i}(A_i)$ и в силу~(\ref{second}) для всех непустых $H_j$ имеем $0 < \eta \le \bigl|H_j \,\, \partial B_{d_i}(A_i)\bigr|.$ Но $\delta \le \varepsilon < \eta,$ поэтому по лемме~\ref{incl} имеем $H_j\subset B_{d_i-\delta}(A_i)$ для всех $j.$ При этом, как отмечалось выше, $H_j\subset K_d.$ Значит, ввиду~(\ref{third_0})  для всех $j$ верно
\begin{equation}\label{third}
H_j\subset K.
\end{equation}

\begin{lem}\label{lem:first}
Для всех $j$ верно $A_j\cap U_{r}(F_j)\subset B_{d_j}(K).$
\end{lem}

\begin{proof}

Если $F_j=\emptyset,$ то $A_j\cap U_{r}(F_j) = \emptyset \subset B_{d_j}(K).$ Пусть теперь $F_j\neq\emptyset.$

Для произвольной $p\in A_j$ имеем $B_{d_j}(p)\cap K_d\neq \emptyset.$ Значит, для любой точки $p\in A_j\cap U_{r}(F_j)$ тоже верно 
\begin{equation}\label{four_00}
B_{d_j}(p)\cap K_d\neq \emptyset.
\end{equation}
Однако согласно лемме~\ref{sum} имеем
\begin{equation}\label{four}
B_{d_j}\bigl(A_j\cap U_{r}(F_j)\bigr)\subset B_{d_j}\bigl(U_{r}(F_j)\bigr) \subset B_{d_j}\bigl(B_{r}(F_j)\bigr) = B_{d_j+r}(F_j).
\end{equation}
Отсюда в силу~(\ref{four}),~(\ref{first}), и~(\ref{third}) справедливо
\begin{equation}\label{four_0}
B_{d_j}\bigl(A_j\cap U_{r}(F_j)\bigr)\cap K_d\subset B_{d_j + r}(F_j)\cap K_d = H_j\subset K.
\end{equation}
Согласно замечанию~\ref{rk:0} верно $\bigcup\limits_{p\,\in \, A_j\cap U_{r}(F_j)} B_{d_j}(p) \subset B_{d_j}\bigl(A_j\cap U_{r}(F_j)\bigr).$ Поэтому в силу~(\ref{four_0}) справедливо 
\begin{equation}\label{five}
\Bigl(\bigcup\limits_{p\,\in \, A_j\cap U_{r}(F_j)} B_{d_j}(p)\Bigr)\cap K_d\subset K.
\end{equation}
Отсюда для любой точки $p\in A_j\cap U_{r}(F_j)$ в силу~(\ref{four_00}) и~(\ref{five}) имеем $\emptyset\neq B_{d_j}(p)\cap K_d\subset K.$ Значит, $p\in B_{d_j}(K).$ Следовательно, $$A_j\cap U_{r}(F_j)\subset B_{d_j}(K).$$ Лемма доказана.

\end{proof}

\begin{lem}\label{lem:second}
Для всех $j$ верно $A_j\subset B_{d_j}(K).$
\end{lem}

\begin{proof}

Для любого $j$ в силу~(\ref{second_0}) имеем 
\begin{equation}\label{six}
A'_j\subset  B_{d_j - \mu}(K_d)
\end{equation}
Согласно~(\ref{third_00}) верно
\begin{equation}\label{six_0}
B_{d_j - \mu}(K_d)\subset B_{d_j - \mu+\varepsilon}(K).
\end{equation}
Но ввиду сделанного выше выбора $0 < \varepsilon < \mu$ справедливо
\begin{equation}\label{six_00}
B_{d_j - \mu+\varepsilon}(K) \subset B_{d_j}(K).
\end{equation}
Отсюда согласно~(\ref{six}),~(\ref{six_0}) и~(\ref{six_00}) получаем 
\begin{equation}\label{six_000}
A'_j \subset B_{d_j}(K).
\end{equation}
При этом по лемме~\ref{lem:first} справедливо 
\begin{equation}\label{six_0000}
A_j\cap U_{r}(F_j)\subset B_{d_j}(K)
\end{equation}
Но ввиду~(\ref{second_00}) для любого $j$ верно 
\begin{equation}\label{six_00000}
A_j = A'_j\cup \bigl(A_j\cap U_{r}(F_j)\bigr).
\end{equation}
Значит, в силу~(\ref{six_000}),~(\ref{six_0000}) и~(\ref{six_00000}) получаем $$A_j\subset B_{d_j}(K).$$ Лемма доказана.

\end{proof}

Заметим, что $K\subset K_d\subset B_{d_j}(A_j)$ для всех $j.$ Значит, по лемме~\ref{lem:second} для всех $j$ имеем 
\begin{equation}\label{seven_0}
d_H(K, A_j)\le d_j.
\end{equation}

\begin{lem}\label{lem:third}
$d_H(K, A_i)\le \max(d_i-\delta, d_i-\mu+\varepsilon) < d_i.$
\end{lem}

\begin{proof}

Напомним, что по предположению $\HP_d(F_i)=\emptyset,$ что по лемме~\ref{empt} эквивалентно $F_i = \emptyset.$ Отсюда согласно~(\ref{second_00}) и замечанию~\ref{rk:2} верно 
\begin{equation}\label{seven}
A'_i=A_i\setminus U_{r}(F_i) = A_i.
\end{equation}
Далее, в силу~(\ref{third_0}) $$K = B_{d_i-\delta}(A_i)\cap K_d\subset B_{d_i-\delta}(A_i).$$ При этом ввиду~(\ref{seven}),~(\ref{six}) и~(\ref{six_0}) верно $$A_i = A'_i\subset B_{d_i - \mu+\varepsilon}(K).$$ Поэтому $d_H(K, A_i)\le \max(d_i-\delta, d_i-\mu+\varepsilon) < d_i,$ так как $0 < \varepsilon < \mu.$ Лемма доказана.

\end{proof}

Следовательно, согласно~(\ref{seven_0}) и лемме~\ref{lem:third} верно $S(A, K) < S(A, K_d).$ Получили противоречие. Теорема доказана.

\end{proof}

\subsection{Устойчивость границы в проблеме Ферма--Штейнера}\label{Saving}

Нам понадобятся обозначения:
\begin{itemize}
	\item $D_i^A$ --- множество всех дискретных точек компакта $A_i\in A$ и $D^A = \bigcup\limits_{i} D_i^A;$
	\item $\HP_d(p, D_i^A) = B_{d_i}(p)\cap K_d,$ где $p\in D_i^A;$
	\item $\HP_d(D_i^A) = \bigcup\limits_{p\in D_i^A} \HP_d(p, D_i^A);$
	\item $\HP_d(D^A) = \bigcup\limits_{i} \HP_d(D_i^A);$
\end{itemize}

Везде, где будет ясно, о какой границе $A$ идёт речь, верхний индекс $A$ писаться не будет, то есть $D_i^A$ и $D^A$ будут заменяться на $D_i$ и $D$ соответственно. 

Пусть дана финитная граница $A=\{A_1,\ldots,A_n\}\subset \mathcal{H}(X).$ Обозначим границу $\{\Conv(A_1),\ldots, \Conv(A_n)\}$ через $A^{\Conv}$. Введём определение устойчивой границы.

\begin{dfn}
Финитную границу $A=\{A_1,\ldots,A_n\}\subset \mathcal{H}(X)$ назовём \textit{устойчивой}, если $S_A=S_{A^{\Conv}},$ иначе --- \textit{неустойчивой}.
\end{dfn}

\begin{rk}\label{remark}
В силу следствия~\ref{dH} для любой финитной границы $A$ верно неравенство $S_A\ge S_{A^{\Conv}}$.
\end{rk}

Напомним, что множество $\bigcap\limits_{i=1}^n B_{d_i}\bigl(\Conv(A_i)\bigr)$ в разделе~\ref{Convex} было обозначено через $K_d^{\Conv}.$ Следующее утверждение было сделано в работе~\cite{Tropin}, однако здесь приводится его альтернативное доказательство на основе следствия~\ref{dH}.

\begin{ass}[Необходимое условие устойчивости]\label{New_K_d}
Если граница $A = \{A_1, \ldots, A_n\}$ устойчива, то $K_d^{\Conv}$ является максимальным компактом Штейнера для $A^{\Conv}.$
\end{ass}

\begin{proof}

Пусть $A = \{A_1, \ldots, A_n\}$ устойчива. Тогда в силу $S_A = S_{A^{\Conv}}$ и следствия~\ref{dH} мы получаем, что $d_H\bigl(\Conv(A_i), K_d^{\Conv}\bigr) = d_i$ для всех $i.$ Значит, так как $K_d^{\Conv} = \bigcap\limits_{i=1}^n B_{d_i}\bigl(\Conv(A_i)\bigr),$ мы получаем, что $K_d^{\Conv}$ --- максимальный компакт Штейнера для $A^{\Conv}.$

\end{proof}

Положим $$U_d^{\Conv} = \bigcap\limits_{i=1}^n U_{d_i}\bigl(\Conv(A_i)\bigr).$$ Заметим, что так как внутренность пересечения равна пересечению внутренностей, то $\Int K_d^{\Conv} = U_d^{\Conv}$.

Напомним, что согласно теореме~\ref{MS_2} в случае финитной границы $A$ и пространства $X$ со строго выпуклой нормой для любого класса $\Sigma_d(A)$ по крайней мере в одном граничном компакте $A_i$ найдётся дискретная точка, то есть $\HP_d(D^A)\neq \emptyset.$

Всюду далее в случае финитной границы $A$ количество точек в компакте $A_i\in A$ будет обозначаться через $m_i.$

\begin{thm}[Достаточное условие неустойчивости]\label{DStay}
Пусть норма пространства $X$ строго выпукла, граница $A$ финитна, все $d_i$ положительны и $U_d^{\Conv}\neq \emptyset.$ Тогда граница $A$ является неустойчивой, если существует номер $s$ такой, что 
\begin{equation}\label{thm_cond}
\Bigl(\bigcup\limits_{j=1}^{m_s} \partial B_{d_s}(a_j^s)\Bigr)\cap \HP_d(D^A)\subset U_d^{\Conv}.
\end{equation}
\end{thm}

\begin{proof}

Всюду далее будем считать, что условие~(\ref{thm_cond}) выполнено. Опишем план доказательства. Оно будет построено методом от противного, то есть сначала мы предположим, что граница устойчива. Тогда согласно утверждению~\ref{New_K_d} верно, что $K_d^{\Conv}$ --- максимальный компакт Штейнера для границы $A^{\Conv}.$ На основании этого факта в начале доказательства будет показано, что все точки из $A_s$ не являются далёкими по отношению к  $K_d^{\Conv}$ --- лемма~\ref{suf_cond:0}. Затем отталкиваясь от этого, мы докажем, что и все остальные точки из $\Conv(A_s)$ не являются далёкими по отношению к  $K_d^{\Conv}$. Значит, мы таким образом докажем, что верно $\HP_d\bigl(F_s^{A^{\Conv}}\bigr) = \emptyset$ --- лемма~\ref{suf_cond:1}. Затем мы покажем, что при $i\neq s$ для любой точки $a\in A_i\cap F_i^{A^{\Conv}}$ верно $\HP_d\bigl(a, F_i^{A^{\Conv}}\bigr) \cap \partial B_{d_s}\bigl(\Conv(A_s)\bigr) = \emptyset$ --- лемма~\ref{suf_cond:2}. И наконец, докажем, что $\HP_d\bigl(F_i^{A^{\Conv}}\bigr)\cap \partial B_{d_s}\bigl(\Conv(A_s)\bigr) = \emptyset$ при $i\neq s$ --- лемма~\ref{suf_cond:3}. В связи с этим мы придём к противоречию с теоремой~\ref{pHP}. Отсюда получим, что граница $A$ является неустойчивой.

\begin{lem}\label{suf_cond:0}
$A_s\cap F_s^{A^{\Conv}} = \emptyset.$
\end{lem}

\begin{proof}

Пусть $a\in A_s.$ Рассмотрим два случая: $B_{d_s}(a)\cap K_d$ конечно и бесконечно. Отметим, что множество $B_{d_s}(a)\cap K_d$ не может быть пустым, так как $A_s\subset B_{d_s}(K_d).$

Если $B_{d_s}(a)\cap K_d$ конечно, то $a\in D_s^A$ и $B_{d_s}(a)\cap K_d = \HP_d(a, D_s^A).$  Значит, $\partial B_{d_s}(a)\cap K_d\subset  \HP_d(a, D_s^A).$ Но $ \HP_d(a, D_s^A)\subset K_d.$ Отсюда $\partial B_{d_s}(a)\cap K_d = \partial B_{d_s}(a)\cap  \HP_d(a, D_s^A).$ Тогда согласно условию~(\ref{thm_cond}) справедливо $\partial B_{d_s}(a)\cap K_d\subset U_d^{\Conv}.$ 

Если $\partial B_{d_s}(a)\cap K_d = \emptyset,$ то $U_{d_s}(a)\cap K_d\neq \emptyset,$ так как $B_{d_s}(a)\cap K_d\neq \emptyset.$ Но тогда ввиду $K_d\subset K_d^{\Conv}$ справедливо $U_{d_s}(a)\cap K_d^{\Conv}\neq \emptyset$ и, значит, $a\notin F_s^{A^{\Conv}}.$

Теперь рассмотрим ситуацию, когда $\partial B_{d_s}(a)\cap K_d \neq \emptyset.$ Напомним, что мы имеем $\partial B_{d_s}(a)\cap K_d\subset U_d^{\Conv}.$ Поэтому в данном случае имеем $\partial B_{d_s}(a)\cap K_d\cap U_d^{\Conv}\neq \emptyset.$ Отсюда справедливо $\partial B_{d_s}(a)\cap U_d^{\Conv}\neq \emptyset.$ Следовательно, так как $U_d^{\Conv}$ открыто, по лемме~\ref{interior} получаем $U_{d_s}(a)\cap U_d^{\Conv}\neq \emptyset,$ и, значит, $U_{d_s}(a)\cap K_d^{\Conv}\neq \emptyset$. Отсюда $a\notin F_s^{A^{\Conv}}.$

Рассмотрим теперь второй случай: $B_{d_s}(a)\cap K_d$ бесконечно. Но тогда $B_{d_s}(a)\cap K_d^{\Conv}$ тоже бесконечно, так как $K_d\subset K_d^{\Conv}.$ Отсюда в силу строгой выпуклости нормы пространства $X$ и выпуклости компакта $K_d^{\Conv}$ множество $U_{d_s}(a)\cap K_d^{\Conv}$ также бесконечно. Значит, тоже получаем $a\notin F_s^{A^{\Conv}}.$

Следовательно, для любой $a\in A_s$ верно $a\notin F_s^{A^{\Conv}}.$ Значит, $A_s\cap F_s^{A^{\Conv}} = \emptyset.$ Лемма доказана.

\end{proof}

\begin{lem}\label{suf_cond:1}
$F_s^{A^{\Conv}} = \emptyset$ и, значит, $\HP_d\bigl(F_s^{A^{\Conv}}\bigr) = \emptyset.$
\end{lem}

\begin{proof}

Рассмотрим теперь точку $a\in \Conv(A_s)\setminus A_s$. По свойству выпуклых оболочек для неё можно выбрать такие $\lambda_j \ge 0$, что $\sum\limits_{j=1}^{m_s} \lambda_j a_j^s = a$, где $\sum\limits_{j=1}^{m_s} \lambda_j = 1.$ Так как по лемме~\ref{suf_cond:0} верно $U_{d_s}(a_j^s)\cap K_d^{\Conv}\neq \emptyset$, то для каждой точки $a_j^s$ существует $p_j\in K_d^{\Conv}$ такая, что $||a_j^s-p_j||<d_s$. Обозначим сумму $\sum\limits_{j=1}^{m_s} \lambda_j p_j$ через $p$. Так как $K_d^{\Conv}$ --- выпуклое множество, то $p\in K_d^{\Conv}$.  Имеем $$||a-p||=||\sum\limits_{j=1}^{m_s} \lambda_j(a_j^s-p_j)||\le \sum\limits_{j=1}^{m_s} \lambda_j ||a_j^s-p_j||<d_s.$$ Значит, $U_{d_s}(a)\cap K_d^{\Conv}\neq \emptyset$. Следовательно, для любой точки $a\in \Conv(A_s)\setminus A_s$ верно $a\notin F_s^{A^{\Conv}}.$ Отсюда ввиду леммы~\ref{suf_cond:0} верно $a\notin F_s^{A^{\Conv}}$ для всех точек $a\in \Conv(A_s).$ Поэтому $F_s^{A^{\Conv}} = \emptyset$ и, значит, $\HP_d\bigl(F_s^{A^{\Conv}}\bigr) = \emptyset.$ Лемма доказана.

\end{proof}

\begin{lem}\label{suf_cond:2}
$\HP_d\bigl(a, F_i^{A^{\Conv}}\bigr)\subset U_{d_s}\bigl(\Conv(A_s)\bigr)$ для любой $a\in A_i\cap F_i^{A^{\Conv}}$ при $i\neq s.$
\end{lem}

\begin{proof}

Так как $a\in F_i^{A^{\Conv}},$ то $U_{d_i}(a)\cap K_d^{\Conv}=\emptyset.$ Поэтому в силу строгой выпуклости нормы пространства $X$ и выпуклости компакта $K_d^{\Conv}$ множество $B_{d_i}(a)\cap K_d^{\Conv}$ одноточечно. Но тогда пересечение $B_{d_i}(a)\cap K_d$ тоже одноточечно как непустое подмножество $B_{d_i}(a)\cap K_d^{\Conv}$, то есть 
\begin{equation}\label{eql}
\HP_d\bigl(a, F_i^{A^{\Conv}}\bigr)=B_{d_i}(a)\cap K_d^{\Conv}=B_{d_i}(a)\cap K_d=\HP_d(a,D_i^A).
\end{equation}
По условию~(\ref{thm_cond}) и в силу $U_d^{\Conv} = \bigcap\limits_{j=1}^{n} U_{d_j}\bigl(\Conv(A_j)\bigr)$ справедливо 
\begin{equation}\label{one_eq}
\bigcup\limits_{j=1}^{m_s} \partial B_{d_s}(a_j^s)\cap \HP_d(a, D_i^A)\subset U_d^{\Conv}\subset U_{d_s}\bigl(\Conv(A_s)\bigr),
\end{equation}
При этом так как $\HP_d(a, D_i^A)\subset K_d\subset B_{d_s}(A_s),$ то 
\begin{equation}\label{two_eq}
\HP_d(a, D_i^A)\setminus \bigcup\limits_{j=1}^{m_s} \partial B_{d_s}(a_j^s)\subset U_{d_s}(A_s)\subset U_{d_s}\bigl(\Conv(A_s)\bigr).
\end{equation}
Отсюда в силу~(\ref{one_eq}) и~(\ref{two_eq}) получаем $\HP_d(a, D_i^A)\subset U_{d_s}\bigl(\Conv(A_s)\bigr)$. Значит, ввиду~(\ref{eql}) $$\HP_d\bigl(a, F_i^{A^{\Conv}}\bigr)\subset U_{d_s}\bigl(\Conv(A_s)\bigr).$$ Лемма доказана.

\end{proof}

\begin{lem}\label{suf_cond}
Пусть $i\neq s,$  $a\in F_i^{A^{\Conv}}$ и для каждой $a_j^i\in A_i$ выполняется $$B_{d_i}(a_j^i)\cap K_d^{\Conv}\cap U_{d_s}\bigl(\Conv(A_s)\bigr)\neq \emptyset.$$ Тогда верно $$\HP_d\bigl(a, F_i^{A^{\Conv}}\bigr)\subset U_{d_s}\bigl(\Conv(A_s)\bigr).$$
\end{lem}

\begin{proof}

Имеем $a\in F_i^{A^{\Conv}}\subset \Conv(A_i).$ Значит, по свойству выпуклых оболочек найдутся такие $\lambda_j \ge 0$ с условием $\sum\limits_{j=1}^{m_i} \lambda_j = 1,$ что $a = \sum\limits_{j=1}^{m_i} \lambda_j a_j^i.$ В таком случае для каждой $a_j^i\in A_i$ возьмём точку $p_j\in B_{d_i}(a_j^i)\cap K_d^{\Conv}\cap U_{d_s}\bigl(\Conv(A_s)\bigr).$ В силу выпуклости $K_d^{\Conv}\cap U_{d_s}\bigl(\Conv(A_s)\bigr)$ верно $$p:=\sum\limits_{j=1}^{m_i} \lambda_j p_j \in K_d^{\Conv}\cap U_{d_s}\bigl(\Conv(A_s)\bigr).$$ Но тогда так как $p_j\in B_{d_i}(a_j^i)$ $$||a-p||=||\sum\limits_{j=1}^{m_i}  \lambda_j(a_j^i-p_j)||\le \sum\limits_{j=1}^{m_i} \lambda_j ||a_j^i-p_j||\le d_i.$$ Значит, справедливо $p\in B_{d_i}(a)\cap K_d^{\Conv}\cap U_{d_s}\bigl(\Conv(A_s)\bigr).$ Но в силу строгой выпуклости нормы пространства $X,$ выпуклости компакта $K_d^{\Conv},$ а также согласно условию $U_{d_i}(a)\cap K_d^{\Conv}=\emptyset$ множество $\HP_d\bigl(a, F_i^{A^{\Conv}}\bigr) = B_{d_i}(a)\cap K_d^{\Conv}$ одноточечно. Значит, $\{p\} = \HP_d\bigl(a, F_i^{A^{\Conv}}\bigr)\subset U_{d_s}\bigl(\Conv(A_s)\bigr).$ Лемма доказана.

\end{proof}

\begin{lem}\label{suf_cond:3}
$\HP_d\bigl(a, F_i^{A^{\Conv}}\bigr)\subset U_{d_s}\bigl(\Conv(A_s)\bigr)$ для любой $a\in F_i^{A^{\Conv}}$ при $i\neq s,$ то есть $\HP_d\bigl(F_i^{A^{\Conv}}\bigr)\subset U_{d_s}\bigl(\Conv(A_s)\bigr).$
\end{lem}

\begin{proof}

Пусть $a\in \Conv(A_i)\setminus A_i$, где $i\neq s,$ и 
\begin{equation}\label{t0}
a\in F_i^{A^{\Conv}},
\end{equation}
то есть $U_{d_i}(a)\cap K_d^{\Conv}=\emptyset.$
 
Нам нужно показать, что $B_{d_i}(a)\cap K_d^{\Conv}\subset U_{d_s}\bigl(\Conv(A_s)\bigr).$ Докажем, что для каждой $a_j^i\in A_i$ множество $B_{d_i}(a_j^i)\cap K_d^{\Conv}\cap U_{d_s}\bigl(\Conv(A_s)\bigr)$ непусто. 

Существует два случая. Первый, когда $U_{d_i}(a_j^i)\cap K_d^{\Conv} = \emptyset,$ то есть $a_j^i\in F_i^{A^{\Conv}}.$ Но тогда по лемме~\ref{suf_cond:2} справедливо $B_{d_i}(a_j^i)\cap K_d^{\Conv} = \HP_d\bigl(a_j^i, F_i^{A^{\Conv}}\bigr)\subset U_{d_s}\bigl(\Conv(A_s)\bigr).$ Поэтому $B_{d_i}(a_j^i)\cap K_d^{\Conv}\cap U_{d_s}\bigl(\Conv(A_s)\bigr)\neq \emptyset,$ так как $B_{d_i}(a_j^i)\cap K_d^{\Conv}\neq \emptyset$ ввиду $a_j^i\in A_i\subset B_{d_i}\bigl(K_d^{\Conv}\bigr)$ согласно определению расстояния Хаусдорфа.

Второй случай, когда $U_{d_i}(a_j^i)\cap K_d^{\Conv} \neq \emptyset,$ то есть $a_j^i\notin F_i^{A^{\Conv}}.$ По условию $U_d^{\Conv}\neq\emptyset,$ поэтому $K_d^{\Conv}$ выпуклое множество с непустой внутренностью. Значит, по лемме~\ref{interior} если $U_{d_i}(a_j^i)\cap \partial K_d^{\Conv} \neq \emptyset,$ то $U_{d_i}(a_j^i)\cap U_d^{\Conv} \neq \emptyset.$ Если же $U_{d_i}(a_j^i)\cap \partial K_d^{\Conv} = \emptyset,$ то всё равно $U_{d_i}(a_j^i)\cap U_d^{\Conv} \neq \emptyset,$ так как $U_{d_i}(a_j^i)\cap K_d^{\Conv} \neq \emptyset$ согласно предположению. Отсюда $U_{d_i}(a_j^i)\cap U_d^{\Conv} \neq \emptyset.$ Но так как $K_d^{\Conv}\subset B_{d_s}\bigl(\Conv(A_s)\bigr),$ то $U_d^{\Conv}\subset U_{d_s}\bigl(\Conv(A_s)\bigr).$ Следовательно, $U_{d_i}(a_j^i)\cap U_d^{\Conv}\cap U_{d_s}\bigl(\Conv(A_s)\bigr) \neq \emptyset,$ и отсюда $B_{d_i}(a_j^i)\cap K_d^{\Conv}\cap U_{d_s}\bigl(\Conv(A_s)\bigr) \neq \emptyset.$

Таким образом, для каждой $a_j^i\in A_i$ $$B_{d_i}(a_j^i)\cap K_d^{\Conv}\cap U_{d_s}\bigl(\Conv(A_s)\bigr)\neq \emptyset.$$ Значит, согласно условию~(\ref{t0}) и в силу леммы~\ref{suf_cond} справедливо $$\HP_d\bigl(a, F_i^{A^{\Conv}}\bigr)\subset U_{d_s}\bigl(\Conv(A_s)\bigr)$$ при $i\neq s.$ Отсюда ввиду произвольности точки $a\in \Conv(A_i)\setminus A_i$ и леммы~\ref{suf_cond:2} получаем $$\HP_d\bigl(F_i^{A^{\Conv}}\bigr)\subset U_{d_s}\bigl(\Conv(A_s)\bigr).$$ Лемма доказана.

\end{proof}

Таким образом, существует номер $s$ такой, что для любой точки $t\in \HP_d\bigl(F^{A^{\Conv}}\bigr)$ верно $t\notin \HP_d\bigl(F_s^{A^{\Conv}}\bigr)$ согласно лемме~\ref{suf_cond:1} и $t\notin \partial B_{d_s}\bigl(\Conv(A_s)\bigr)$ согласно лемме~\ref{suf_cond:3}. Получили противоречие с теоремой~\ref{pHP}. Доказательство теоремы закончено.

\end{proof}

\subsection{Пример неустойчивой границы}\label{Example}

В качестве примера возьмём конфигурацию из работы~\cite{M}, где $A = \{A_1, A_2, A_3\}\subset \mathcal{H}\bigl(\mathbb{R}^2\bigr)$ и $A_i = \{a_i, b_i\}$ для всех $i,$ см. рис.~\ref{Circle :image}.
\begin{figure}[h]
\center{\includegraphics[scale=0.65]
{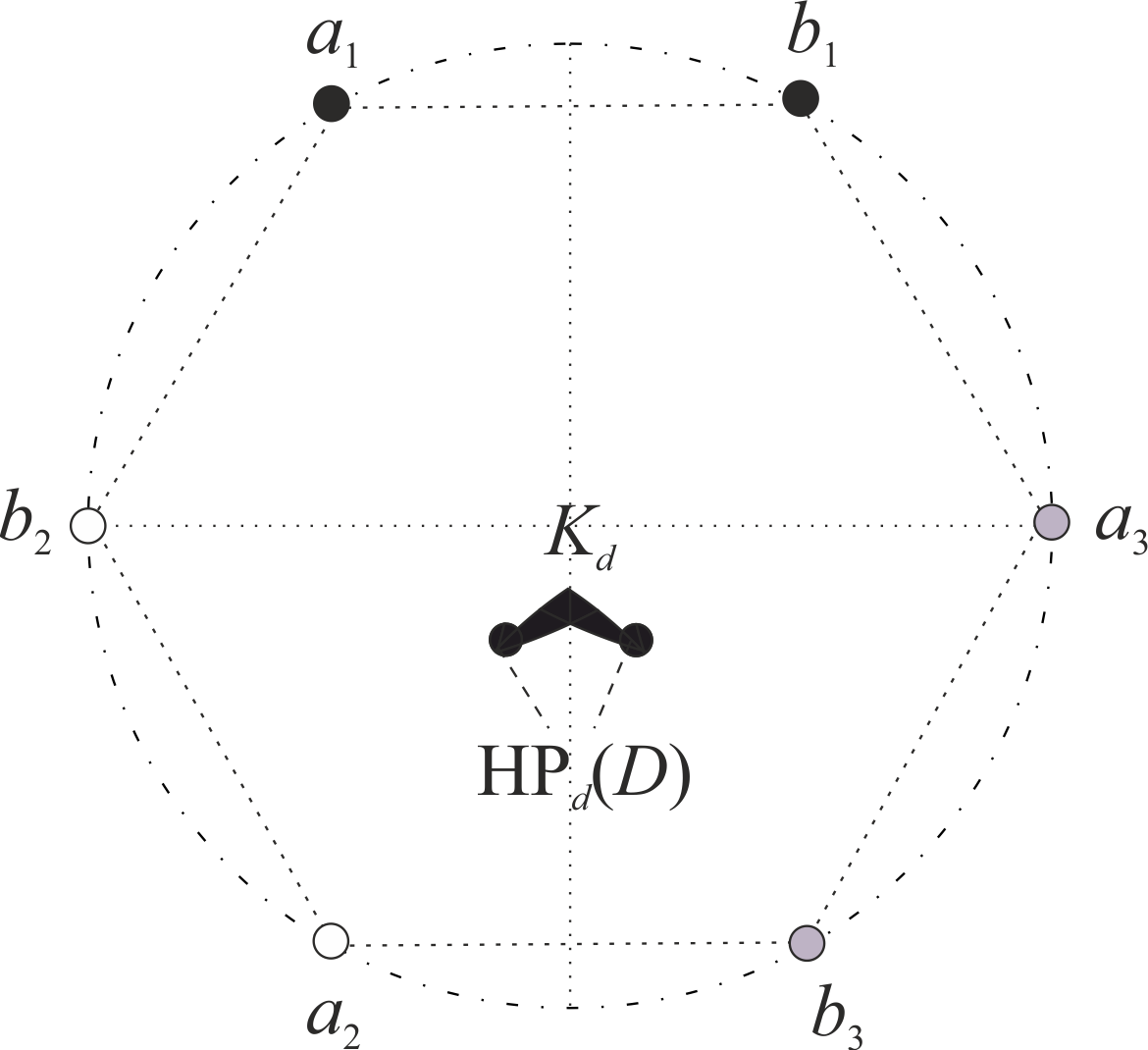}}
\caption{Конфигурация из работы~\cite{M}.}
\label{Circle :image} 
\end{figure}
В данной конфигурации множество точек $\{a_1, b_1, a_2, b_2, a_3, b_3\}$ расположено на единичной окружности с центром в начале координат, и оно является множеством вершин правильного шестиугольника. Координаты точек следующие:
$$a_1 = \Bigl(-\cos(\pi/3); \sin(\pi/3)\Bigr);$$ $$b_1 = \Bigl(\cos(\pi/3); \sin(\pi/3)\Bigr);$$ 
$$a_2 = \Bigl(-\cos(\pi/3); -\sin(\pi/3)\Bigr);$$ $$b_2 = (-1; 0);$$
$$a_3 = (1; 0);$$ $$b_3 = \Bigl(\cos(\pi/3); -\sin(\pi/3)\Bigr).$$

Чёрным цветом на рис.~\ref{Circle :image} изображён максимальный компакт Штейнера $K_d$ одного из трёх классов решений для границы $A.$ Также в границе $K_d$ выделено множество $\HP_d(D),$ которое в данном случае состоит из двух точек и совпадает с минимальным компактом Штейнера $K_\lambda,$ являющимся единственным минимальным компактом в рассматриваемом классе решений, см. работу~\cite{M}. 

Возьмём теперь границу $A^{\Conv} = \bigl\{\Conv(A_1), \Conv(A_2), \Conv(A_3)\}\bigr\},$ см. рис.~\ref{Conv :image}, и обозначим левую точку $\HP_d(D)$ через $p,$ а правую --- через $q.$
\begin{figure}[h]
\center{\includegraphics[scale=0.65]
{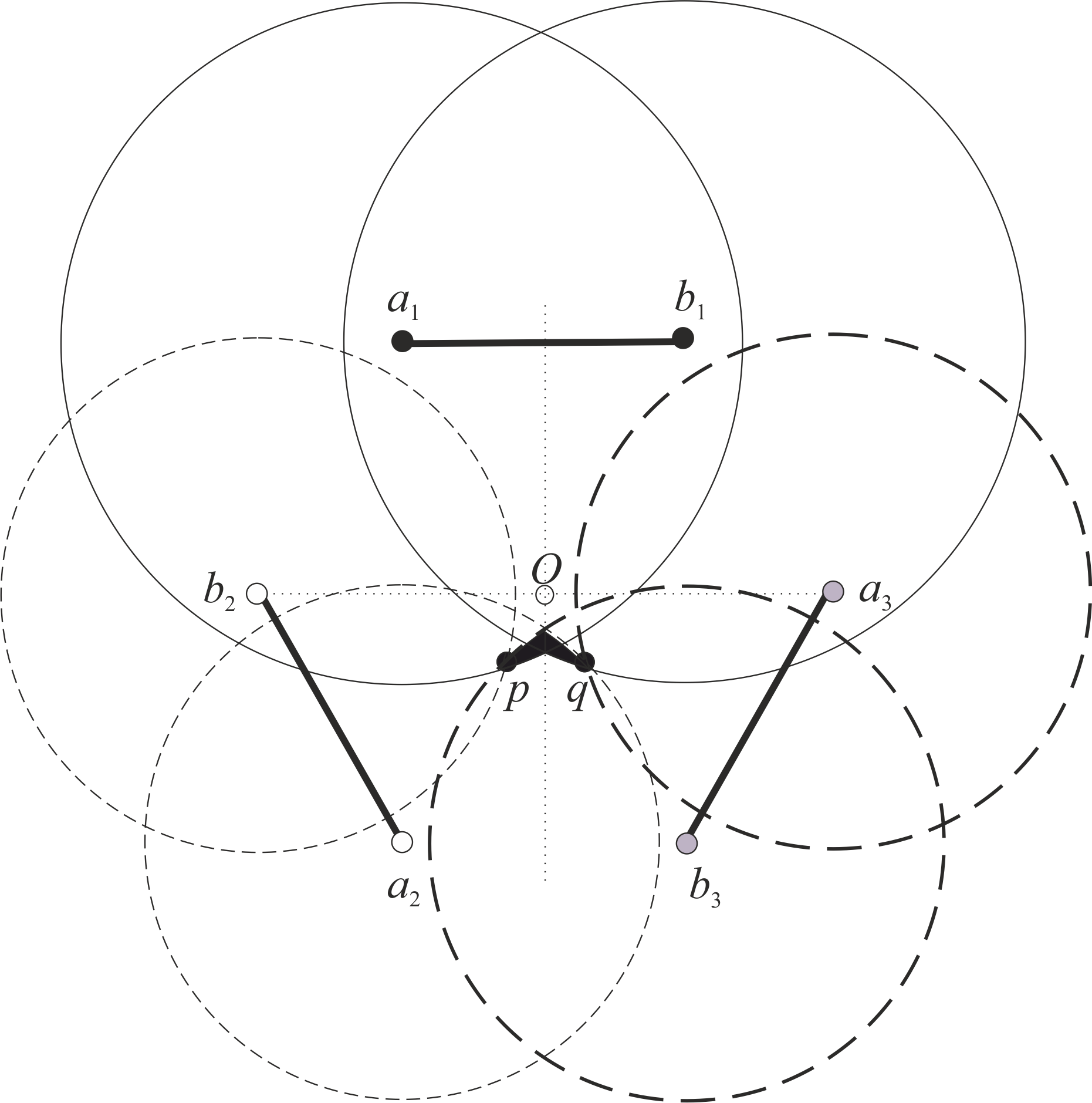}}
\caption{Граница $A^{\Conv}$ и окрестности $U_{d_1}(A_1), U_{d_2}(A_2), U_{d_3}(A_3).$}
\label{Conv :image} 
\end{figure}

Согласно работе~\cite{M} 
\begin{equation}\label{eval}
||O - p|| = \frac{\sqrt{5}-\sqrt{4\sqrt{5}-7}}{4} < 0.5 .
\end{equation} 
Также точка $p$ лежит на отрезке $[O, a_2]$ и угол между $[O, a_2]$ и $[O, b_2]$ равен $\pi/3.$ Значит, в декартовых координатах $$p = \frac{\sqrt{5}-\sqrt{4\sqrt{5}-7}}{4}\cdot \bigl(-\cos(\pi/3); -\sin(\pi/3) \bigr).$$ Точка $q$ располагается зеркально относительно вертикальной оси симметрии, проходящей через точку $O,$ поэтому $$q =\frac{\sqrt{5}-\sqrt{4\sqrt{5}-7}}{4}\cdot \bigl(\cos(\pi/3); -\sin(\pi/3)\bigr).$$

Возьмём компакт $A_1=\{a_1,b_1\}$ и покажем, что выполняется $$\Bigl(\partial B_{d_1}(a_1)\cup \partial B_{d_1}(b_1)\Bigr)\cap \HP_d(D) \subset \bigcap\limits_{i=1}^3 U_{d_i}\bigl(\Conv(A_i)\bigr) = U_d^{\Conv}.$$ Имеем $K_\lambda=\HP_d(D)\subset \partial B_{d_1}(a_1)\cup \partial B_{d_1}(b_1),$ см. рис.~\ref{Conv :image}. Поэтому $$\Bigl(\partial B_{d_1}(a_1)\cup \partial B_{d_1}(b_1)\Bigr)\cap \HP_d(D) = \HP_d(D).$$ Следовательно, нам нужно доказать $\HP_d(D)\subset U_d^{\Conv}.$ Покажем сначала, что $\HP_d(D)\subset U_{d_1}\bigl(\Conv(A_1)\bigr).$ 

В силу~(\ref{eval}) абсцисса точки $a_1$ меньше абсциссы точки $p.$ При этом $p\in B_{d_1}(a_1).$ Значит, расстояние от $p$ до отрезка $[a_1, b_1]$ меньше $d_1.$ Аналогично, абсцисса $b_1$ больше абсциссы $q$ ввиду~(\ref{eval}), и $q\in B_{d_1}(b_1).$ Поэтому расстояние от $q$ до $[a_1, b_1]$ также меньше $d_1.$ Следовательно, $$\HP_d(D) = \{p, q\} \subset U_{d_1}\bigl(\Conv(A_1)\bigr).$$

Далее покажем $\HP_d(D) \subset U_{d_2}\bigl(\Conv(A_2)\bigr).$ Согласно доказанному в работе~\cite{M} имеем $p\in U_{d_2}(a_2),$ поэтому нам надо только показать, что $q \in U_{d_2}\bigl([a_2, b_2]\bigr).$ Для удобства обозначим величину $\frac{\sqrt{5}-\sqrt{4\sqrt{5}-7}}{4}$ через $c.$ Далее замечаем, что косинус угла между отрезками $[a_2, b_2]$ и $[a_2, q]$ равен
$$\alpha:= \frac{(1+c)\cdot \cos(\pi/3) \cdot\bigl(-1 + \cos(\pi/3)\bigr) + (1 - c)\cdot \sin^2(\pi/3)}{(1+c)^2\cdot \cos^2(\pi/3) + (1-c)^2\cdot \sin^2(\pi/3)}.$$
Покажем, что $\alpha > 0.$ Знаменатель положительный, поэтому достаточно рассмотреть только числитель. Имеем
\begin{multline}
(1+c)\cdot \cos(\pi/3) \cdot\bigl(-1 + \cos(\pi/3)\bigr) + (1 - c)\cdot \sin^2(\pi/3) = \\ 
= \cos^2(\pi/3) - \cos(\pi/3) + c\cdot\cos^2(\pi/3) - c\cdot\cos(\pi/3) + \sin^2(\pi/3) - c\cdot\sin^2(\pi/3) = \\
= 1 + c\cdot \cos(2\pi/3) - c\cdot\cos(\pi/3) - \cos(\pi/3) = 1 - 2c\cdot\cos(\pi/3) - \cos(\pi/3) = \\ 
= 1 - (2c + 1)\cdot \cos(\pi/3) = 1 - c - 0.5 = 0.5 - c.
\end{multline}
В силу~(\ref{eval}) получаем $0.5 - c > 0.$ Следовательно, $\alpha > 0.$ Это означает, что угол между отрезками $[a_2, b_2]$ и $[a_2, q]$ меньше $\pi/2.$ При этом $p\in B_{d_2}(a_2).$ Отсюда расстояние от $q$ до отрезка $[a_2, b_2]$ меньше $d_2.$ Значит, $$\HP_d(D) = \{p, q\} \subset U_{d_2}\bigl(\Conv(A_2)\bigr).$$

В силу зеркальной симметрии аналогично доказывается, что $\{p, q\}\subset U_{d_3}\bigl(\Conv(A_3)\bigr).$ Таким образом, мы показали $$\HP_d(D)\subset U_{d_i}\bigl(\Conv(A_i)\bigr)$$ для всех $i,$ см. рис.~\ref{K_d_Conv}. 
\begin{figure}[h]
\center{\includegraphics[scale=0.65]
{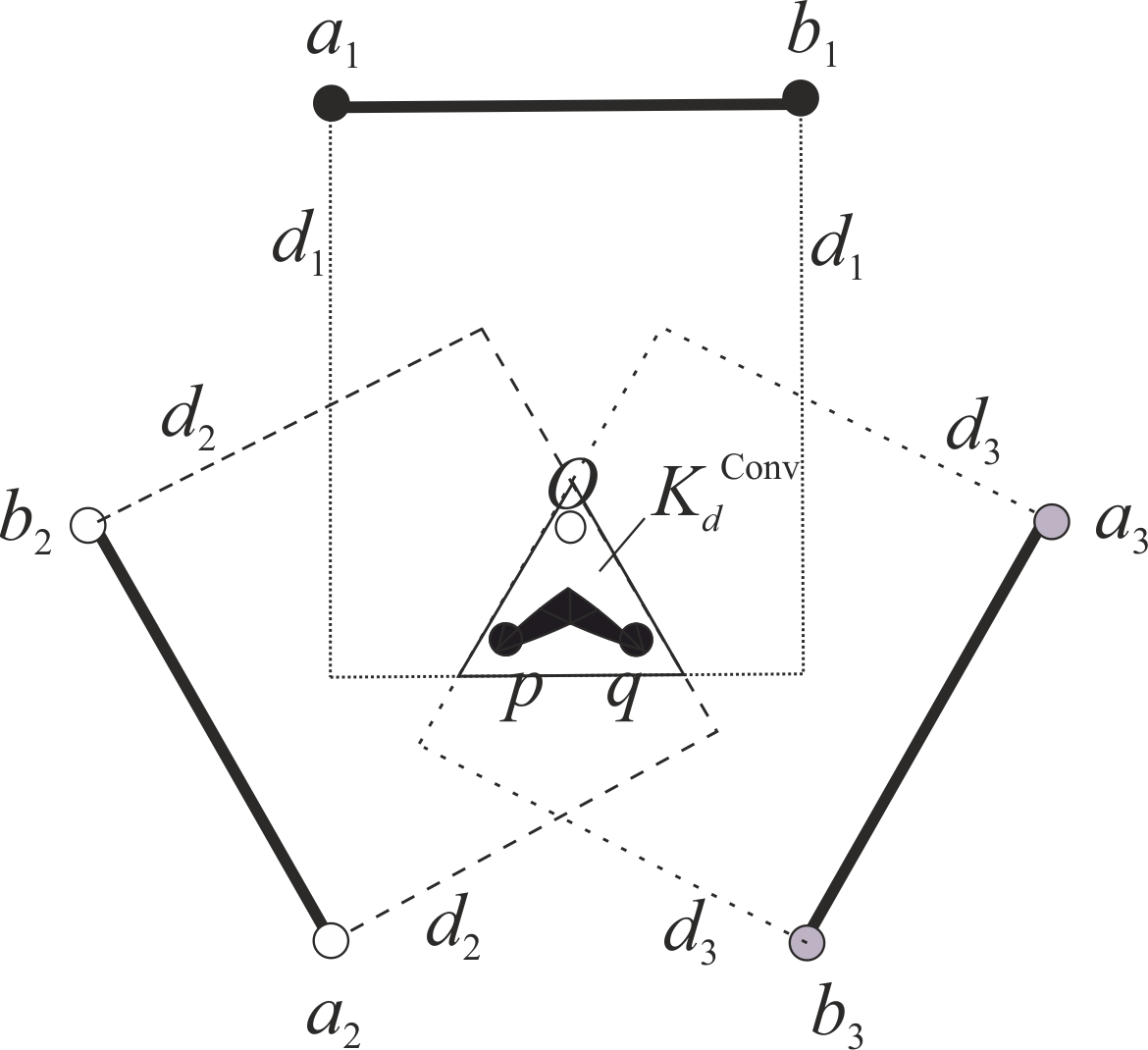}}
\caption{$\HP_d(D)\subset U_d^{\Conv} = \Int K_d^{\Conv}$.}
\label{K_d_Conv} 
\end{figure}
Следовательно, мы нашли компакт $A_1$ такой, что $$\emptyset\neq \Bigl(\partial B_{d_1}(a_1)\cup \partial B_{d_1}(b_1)\Bigr)\cap \HP_d(D) = \HP_d(D)\subset \bigcap\limits_{i=1}^3 U_{d_i}\bigl(\Conv(A_i)\bigr) = U_d^{\Conv}.$$ Также отметим, что $d_1, d_2, d_3$ положительны и евклидова норма $\mathbb{R}^2$ является строго выпуклой. Значит, выполнены все условия теоремы~\ref{DStay}. Отсюда данная граница неустойчива.

\label{end}

\end{document}